\documentclass{article}
\usepackage{epsfig,amsmath,mathenv,amsthm,amssymb,graphics,verbatim,amsfonts}

\usepackage{makeidx}

\theoremstyle{plain}
\newtheorem{thm}{Theorem}[section]

\newtheorem{lem}[thm]{Lemma}
\newtheorem{ex}[thm]{Example}
\newtheorem{fact}[thm]{Fact}

\theoremstyle{definition}
\newtheorem{defn}[thm]{Definition}

\newcommand{\mathsym}[1]{{}}

\makeindex

\begin{document}
 
\date{}
\title{Introduction to Potential Theory via Applications}
\author{Christian Kuehn}

\maketitle
\thispagestyle{empty}

\begin{abstract}

We introduce the basic concepts related to subharmonic functions and potentials, mainly for the case of the complex plane and prove the Riesz decomposition theorem. Beyond the elementary facts of the theory we deviate slightly from the usual path of exposition and introduce further concepts alongside with applications.\\

We cover the Dirichlet problem in detail and illustrate the relations between potential theory and probability by considering harmonic measure and its relation to Brownian motion. Furthermore Green's function is introduced and an application to growth of polynomials is given. Equilibrium measures are motivated by their original development in physics and we end with a brief discussion of capacity and its relation to Hausdorff measure.\\

We hope that the reader, who is familiar with the main elements of real analysis, complex analysis, measure theory and some probability theory benefits from these notes. No new results are presented but we hope that the style of presentation enables the reader to understand quickly the basic ideas of potential theory and how it can be used in different contexts.\\

The notes can also be used for a short course on potential theory. Therefore the required prerequisites are described in the Appendix. References are given where expositions and details can be found; roughly speaking, familiarity with the basic foundations of real and complex analysis should suffice to proceed without any background reading.  

\end{abstract}

\tableofcontents

\newpage

\setcounter{page}{1}

\section{Introduction}

The main goal of this paper is twofold, the first part is an attempt to present the basic results of potential theory focusing on the case for two dimensions. The main objects in this section are harmonic and subharmonic functions as well as potentials. The Riesz decomposition theorem is the central result in this context relating the three classes of functions and justifying that studying subharmonic functions is closely related to studying potentials. The treatment of the material is fairly standard and can be found in any major textbook in one way or the other, but it turns out to be essential.\\

The second part deviates in many ways from classical paths usually persued in the presentation of the subject. One is either confronted with a purely theoretical introduction barely mentioning applications to other areas of mathematics at all or just as a seperated chapter in the end of a book. Even if one encounters applications such as the solution of the Dirichlet problem, the presentation tends to be obscured with technicalities to prove very general results, although a slightly restricted viewpoint would improve the presentation considerably. On the other hand many texts deal e.g. with subharmonic functions only from the viewpoint to use them as tools for a specialized field of application. The only case, where both aspects are persued, are monumental monographs, which have the big disadvantage that it becomes hard to seperate the important results and concepts from digressions.\\

We attempt in the second part of this paper to introduce the fundamentals of potential theory alongside simple applications. The idea is to convey the relevant points together with an outlook, where they are really powerful, namely in different areas of mathematical analysis apart from potential theory. We have to give up full generality, which - considering space limitations and the goal of these notes - seems to be adequate bearing in mind that we mainly deal with potential theory in the complex plane anyway.   

\section{Part I - Basic Theory}

Our discussion of harmonic and subharmonic functions as well as potentials is mainly based on \cite{4} with parts taken from \cite{3}, \cite{18} and \cite{5}. Although we are going to develop most of the theory in the complex plane it seems to be appropiate to start with a more general setup, thereby illustrating that most concepts can be generalized from two to more dimensions.

\subsection{Harmonic functions}
Before we can begin to introduce 'subharmonic' functions it is helpful to recall some basic notions and results for harmonic functions. Let $B(x,r)\subset \mathbb{R}^{n}$ denote the ball centered at $x$ of radius $r$. Let $m^{(n)}$ denote the $n$-dimensional Lebesque measure and $\sigma^{(n-1)}$ denote the induced surface measure. 

\begin{defn} 
A function $h:\mathbb{R}^n\rightarrow \mathbb{R}$ is called \textbf{\index{harmonic}harmonic} at $x$ if there exists an $r>0$ such that it satisfies 
\begin{eqnarray}
h(x)=\frac{1}{m^{(n)}(B(x,r))}\int_{B(x,r)}h(y)dm^{(n)}(y)
\end{eqnarray}
\end{defn}

In particular, the definition means that $h$ is harmonic if its value at a point $x$ can be locally expressed as an average. We note in passing the following important formula

\begin{fact}[\index{integration in polar coordinates}Integration in Polar Coordinates, \cite{3}] Let $f:\mathbb{R}^n\rightarrow\mathbb{R}$ be a suitably integrable function, then
\begin{eqnarray}
\int_{B(x,R)}f(y)dm^{(n)}(y)=\int_0^R \int_{\partial B(0,1)}f(x+rt)r^{n-1}d\sigma^{(n-1)}(t) dr
\end{eqnarray}
\end{fact}

We are going to use this change to polar coordinates for integration from now on without further notice. If $h$ is harmonic at $x$, we observe that

\begin{eqnarray}
h(x)=\frac{1}{m^{(n)}(B(x,r))}\int_{B(x,r)}h(y)dm^{(n)}(y)=\frac{n}{r^n \sigma^{(n-1)}(\partial B(x,r))}\int_{0}^{r}\int_{\partial B(x,s)}h(y)s^{n-1}d\sigma^{(n-1)}(y)ds \notag\\
\Rightarrow \quad r^{n}h(x)=n\int_0^r\frac{1}{\sigma^{(n-1)}(\partial B(x,r))}\int_{\partial B(x,s)}s^{n-1}h(y)d\sigma^{(n-1)}(y)ds\notag
\end{eqnarray}   

and upon taking derivatives with respect to $r$ on both sides it follows that 

\begin{eqnarray}
nr^{n-1}h(x)&=&\frac{r^{n-1}n}{\sigma^{(n-1)}(\partial B(x,r))}\int_{\partial B(x,r)}h(y)d\sigma^{(n-1)}(y)\notag \\
h(x)&=&\frac{1}{\sigma^{(n-1)}(\partial B(x,r))}\int_{\partial B(x,r)}h(y)d\sigma^{(n-1)}(y)\notag	
\end{eqnarray}

where the last expression gives $h$ as an average over the boundary of the ball. Since the reverse calculation is equally easy we see that we can define $h$ to be harmonic either via a so-called 'space-mean' (i.e. average over $B(x,r)$) or a 'surface-mean' (i.e. average over $\partial B(x,r)$) and that both definitions are equivalent. For convenience we introduce the following notation 

\begin{defn} We define the \textbf{\index{space-mean}space-mean} $B^{r}_{f}(x)$ and the \textbf{\index{surface-mean}surface-mean} $S^{r}_{f}(x)$ of $f$ by
\begin{eqnarray}
B^{r}_{f}(x)&=&\frac{1}{m^{(n)}(B(x,r))}\int_{B(x,r)}f(y)dm^{(n)}(y)	\\
S^{r}_{f}(x)&=&\frac{1}{\sigma^{(n-1)}(\partial B(x,r))}\int_{\partial B(x,r)}f(y)d\sigma^{(n-1)}(y)	
\end{eqnarray}
where $f:\mathbb{R}^{n}\rightarrow \mathbb{R}$ is a suitably integrable function. 
\end{defn}

As another convention for notation we should from now on drop the superscripts for the Lebesque measure and the associated surface measure if there is no possibility of confusion; also we write $dm^{(1)}(x)=dx$. Having defined harmonic functions we have a few trivial examples

\begin{ex}
\label{ex:haex} 
Let $p:\mathbb{R}^2\rightarrow\mathbb{R}$ be defined by $p(x,y)=x$, then $p$ is obviously harmonic. For example if we check harmonicity at $(x,y)=(0,0)$ we get

\begin{eqnarray}
B^{1}_{p}((0,0))=\frac{1}{\pi^2}\int_{B((0,0),1)}pdm=\frac{1}{\pi^2}\int_0^{2\pi}\int_0^1 r^2 \cos(\theta)d\theta dr=0=p(0,0)\notag
\end{eqnarray}

All others values are similarly easy to check. As a second example consider $f:\mathbb{R}^2\rightarrow\mathbb{R}$ defined by $f(x,y)=x^2-y^2$. Again checking at $(x,y)=(0,0)$ we have

\begin{eqnarray}
B^{1}_{f}((0,0))&=&\frac{1}{\pi^2}\int_{B((0,0),1)}f(x,y)dm=\frac{1}{\pi^2}\int_0^{2\pi}\int_0^1 \left(r^2 \cos^2(\theta)-r^2 \sin^2(\theta)\right) r d\theta dr\notag\\
&=&\frac{1}{3\pi^2}\int_0^{2\pi} \left(\cos^2(\theta)-\sin^2(\theta)\right)d\theta=\frac{1}{3\pi^2}\left[\cos(\theta)\sin(\theta)\right]_0^{2\pi}=0=f(0,0)\notag
\end{eqnarray}
With a little bit more work it is easy to see that $f$ is also harmonic in the whole plane. That it is no coincidence that $f,p\in C^{\infty}(\mathbb{R}^2)$ is confirmed by the following result
\end{ex}

\begin{thm}[Harmonic $\Rightarrow$ Smooth]  
\label{thm:hasmooth}
Let $U$ be an open set in $\mathbb{R}^n$ and let $h$ be harmonic on $U$ then $h\in C^{\infty}(U)$.  
\end{thm}

Before proving the result it is helpful to remind oneself of the following

\begin{defn}
\label{defn:ai}
Let $\chi:\mathbb{R}^n\rightarrow\mathbb{R}$ be a function having the following properties\\

\begin{tabular}{ll}
(smoothness) & $\chi \in C^{\infty}(\mathbb{R}^n)$\\
(positivity) & $\chi \geq 0$\\
(support) & $supp(\chi)\subset B(0,1)$\\
(radial function) & $\chi(x)=\chi(|x|)$\\
(unit integral) & $\int_{\mathbb{R}^n}\chi dm=1$\\		
\end{tabular} 
\newline

then for $\delta>0$ we define a new family of functions by $\chi_{\delta}(x)=\frac{1}{\delta^n}\chi\left(\frac{x}{\delta}\right)$. Clearly $\chi_{\delta}$ satisfies four properties as given and we have $supp(\chi_{\delta})\subseteq B(0,\delta)$. A family of functions $\{\chi_{\delta}:\delta>0\}$ is called an \textbf{approximation to the identity} or briefly \textbf{\index{approximate identity}approximate identity}.
\end{defn} 

Note that sometimes one also refers to the function $\chi$ as the aproximate identity as it completely determines the relevant family of functions. That it is indeed possible to find a suitable $\chi(x)$ can be immediately seen by considering 
\begin{eqnarray}
\chi(x)=\begin {cases} 
C\cdot exp\left(-\frac{1}{1-|x|^2}\right) & \text{if $x\in B(0,1)$},\\
0 & \text{if $x\not\in B(0,1)$}.
\end {cases}
\end{eqnarray}

where $|.|$ denotes the usual Euclidean norm and the constant $C>0$ is chosen such that the 'unit integral' property is satisfied. One of the main uses of approximate identities is to convolve them with non-smooth functions as we shall see later in more detail and therefore we note the following result in passing

\begin{fact}[\cite{1}]
\label{fact:conv}
Let $U$ be an open subset of $\mathbb{R}^n$. Then if $f\in L^1_{loc}(U)$, i.e. $f$ is measurable and integrable on any compact set $K\subset U$, and $g\in C^k(\mathbb{R}^n)$ with $supp(g)\subset B(0,r)$ then the \textbf{\index{convolution}convolution} 
\begin{eqnarray}
(f\ast g)(x)=\int_{\mathbb{R}^n} f(x-y)g(y)dm(y)
\end{eqnarray}
is in $C^k(U_{r})$ where $U_{r}=\{x\in U : d(x,\partial U)>r\}$.
\end{fact} 

Having the definition of an approximate identity the proof of Theorem \ref{thm:hasmooth} is a direct computation.

\begin{proof}
(of Theorem \ref{thm:hasmooth}) So assume that $h$ is harmonic on $U$ and let $\chi_{\delta}$ be an approximate identity. We use the notation $\chi_{\delta |.|}(r)=\chi(r)=\chi(|x|)$. Therefore it follows

\begin{eqnarray}
(\chi_{\delta}\ast h)(x)&=& \int_{B(0,\delta)}h(x-y)\chi_\delta(y)dm(y)=\int_0^\delta \chi_{\delta |.|}(r)\int_{\partial B(0,1)}h(x+rt)d\sigma(t) r^{n-1}dr \notag\\
&=& \sigma(\partial B(0,1))\int_0^\delta S^r_h(x)\chi_{\delta |.|}(r)r^{n-1}dr\notag\\
&=&h(x) \sigma(\partial B(0,1))\int_0^\delta \chi_{\delta |.|}(r)r^{n-1}dr\notag\\
&=&h(x)\int_{B(0,\delta)} \chi_{\delta}(y)dm(y)=h(x)\notag
\end{eqnarray}

But by the preceeding Fact \ref{fact:conv} we have that the left hand side of the equation is in $C^{\infty}$ as required.
\end{proof}

The differentiability property justifies the statement of the following fact we believe the reader is certainly aware of

\begin{fact}[\cite{2}] Let $U\subseteq \mathbb{R}^n$, where U is open. Then $h:U\rightarrow\mathbb{R}$ is harmonic (at $y\in U$) if and only if $\Delta h=\sum_{k=1}^{n}\frac{\partial^2 h}{\partial x_k^2}=0$ (at $y\in U$).
\end{fact}

In addition to smoothness, harmonic functions also have other 'restrictions' they have to satisfy. The following result is of fundamental importance

\begin{thm}[\index{Identity principle}Identity principle - harmonic version] 
\label{thm:idha}
Let $g,h$ be harmonic on an open, connected set $U\subseteq \mathbb{R}^n$. Suppose $h=g$ on an open, non-empty set $V\subset U$ then $h=g$ identically on $U$.
\end{thm}

\begin{proof}
By considering $h-g$ we can without loss of generality reduce to the case when $g=0$ on $V$. Let $x\in V$ then $D^{\alpha}h(x)=0$ for all multi-indices $\alpha$ (for notation see \cite{1} or \cite{2}), i.e. all derivatives vanish identically as well on $V$. Then we define
\begin{eqnarray}
A=\{x\in U: D^{\alpha}h(x)=0 \text{ for all } \alpha\} \qquad \text{and} \qquad B=U\backslash A \notag
\end{eqnarray} 
If $x_0\in B$ there exists some multi-index $\alpha$ such that $D^\alpha h(x_0)\neq 0$ and by continuity of all derivatives of $h$ - see Theorem \ref{thm:hasmooth} - this implies that there is also an open neighborhood of $x_0$ where $D^\alpha h(x)\neq 0$, hence $B$ is open. If $x_0\in A$ then we observe that by using a Taylor expansion of $h$ around $x_0$ that $h=0$ in some neighborhood of $x_0$ therefore A is open. But $A\cup B=U$, $A\cap B=\emptyset$ and $U$ is connected and therefore either $B=\emptyset$ or $A=\emptyset$. Since $h=0$ on $V$ we obtain that $A$ is non-empty and therefore $h=0$ on $U$.
\end{proof}    

The identity principle places a very strong requirement on the class of harmonic functions since the result works for any open, connected set $V$. The situation is obviously reminiscent of holomorphic functions for which an even stronger result holds. Therefore we can give easier proof of the identity principle for the case $\mathbb{R}^2$. Note that we are going to identify henceforth $\mathbb{R}^2$ with the complex plane $\mathbb{C}$. 

\begin{thm}[\index{Identity Principle}Identity Principle - harmonic version, $n=2$]
\label{thm:idc}
Let $h$ be harmonic on a domain $D\subseteq \mathbb{C}$. Suppose $h=0$ on an open, non-empty set $V\subset D$, then $h=0$ identically on $D$.
\end{thm} 

\begin{proof}
Define $g=h_x-ih_y$ where subscripts denote partial derivatives. Since $h$ is harmonic it follows that $h_{xx}+h_{yy}=0$. Also as $h\in C^{\infty}$ partial derivatives certainly commute, so $h_{xy}=h_{yx}$ and we obtain
\begin{eqnarray}
Re(g)_x=h_{xx}=-h_{yy}=Im(g)_y \qquad and \qquad Im(g)_x=-h_{yx}=-h_{xy}=-Re(g)_y\notag
\end{eqnarray}
So $g$ satisfies the Cauchy-Riemann equations and is therefore holomorphic. Since $h=0$ on $V$ it follows that $g=0$ on $V$ as well and the identity principle for holomorphic functions (!) gives that $g=0$ on $D$. Therefore $h_x=0=h_y$ and we conclude that $h$ is constant and therefore identically $0$ as it vanishes on $V$.  
\end{proof}

The result shows that the machinery of complex analysis turns out to be useful for analyzing the case of planar harmonic functions. The same turns out to be true for subharmonic functions as well and we are going to focus from now one on this case, pointing out differences to higher dimensions when appropiate. We need some more relations between harmonic and holomorphic functions.

\begin{lem}
\label{lem:ref} 
If $h$ is harmonic on a simply-connected domain $D$ in $\mathbb{C}$, then $h=Re(f)$ for some holomorphic function $f$, where $f$ is unique up to a constant.
\end{lem}

\begin{proof}
Uniqueness is easy since if we have $h=Re(f)$ for $f=h+ik$ then by using the Cauchy Riemann equations $f'=h_x+ik_x=h_x-ih_y$ and so the derivative is uniquely determined by $h$ and therefore $f$ is unique up to a constant. Then we again consider $g=h_x-ih_y$ (see \ref{thm:idc}), which is holomorphic on $D$. Now fix $z_0$ in $D$ and define
\begin{eqnarray}
f(z)=h(z_0)+\int_{z_{0}}^z g(w)dw\notag
\end{eqnarray}  
where the obvious identification for $\mathbb{R}^2$ and $\mathbb{C}$ has been used. Note that the definition of $f$ as an integral is a priori path-dependent, but since $D$ is simply-connected, Cauchy's theorem ensures that we actually have path-independence. The remaining part of the proof amounts to checking that the construction of $f$ was correct. So set $\tilde{h}=Re(f)$. Since we have $f'=g=h_x-ih_y$ it follows that
\begin{eqnarray}
\tilde{h}_x-i\tilde{h}_y=f'=h_x-ih_y\notag
\end{eqnarray} 
Therefore the partial derivatives of $h$ and $\tilde{h}$ coincide and this implies $h-\tilde{h}=constant$. But for $z=z_0$ we have $h(z_0)=\tilde{h}(z_0)$ and therefore $h=\tilde{h}=Re(f)$.
\end{proof}

Basically this shows a characterization of harmonic functions in the plane, namely they are locally real parts of holomorphic functions. We are now ready to prove another strong 'constraint' on harmonic functions. Note that in the case of a planar domain we are going to agree on the convention to take the closure of a set with respect to the Riemann sphere $\mathbb{C}_\infty=\mathbb{C}\cup \{\infty\}$ for the rest of this paper.  

\begin{thm}[\index{Maximum Principle - harmonic version}Maximum Principle - harmonic version, n=2]
\label{thm:maxpri2}
Let $h$ be harmonic on a domain $D$ in $\mathbb{C}$, then
\begin{enumerate}
	\item if $h$ attains a local maximum on $D$ then $h$ is constant.
	\item if $h$ extends continuously to $\overline{D}$ and $h\leq 0$ on $\partial D$ then $h\leq 0$ on $D$.
\end{enumerate}
\end{thm}

\begin{proof}
For the first part, assumme that $h(w)$ is locally a maximum for $h$. Hence there exists $r>0$ such that $h(z)\leq h(w)$ for $z$ in some disc $\Delta(w,r)=\{z:|z-w|<r\}$. By Lemma \ref{lem:ref} there exists $f$ holomorphic on $\Delta(w,r)$ such that $h=Re(f)$. Therefore $|e^{f}|$ attains a local maximum at $w$ since $h$ does. Now the maximum principle for holomorphic functions (!) implies that $e^f$ is constant on $\Delta(w,r)$, therefore $h$ is constant on $\Delta(w,r)$. Finally the identity principle (Theorem \ref{thm:idc}) yields that $h$ is constant on $D$. This concludes the first part of the maximum principle.\\

For the second part we have that $h\leq 0$ on $\partial D$ and $h$ extends continuously to the boundary, then by compactness of $\overline{D}$, $h$ attains a maximum at some $w\in\overline{D}$. Now either $w\in\partial D$ in which case $h(z)\leq 0$ for $z\in D$ follows immediately, so suppose $w\not\in\partial D$. Then by the first part (see \textit{1.}) $h$ is constant on $D$. By continuous extension to the boundary and $h(z)\leq 0$ on $\partial D$ we obtain $h\leq 0$ on $D$.    
\end{proof}

Note that the first part of the maximum principle can be extended to higher dimensions.

\begin{fact}[\cite{3}] 
\label{thm:maxpri3}
Let $h$ be harmonic on an open connected set $U\subseteq \mathbb{R}^n$; if $h$ attains a local maximum on $U$, then $h$ is constant. 
\end{fact} 

We conclude the discussion with a trivial result, which generates even more examples of harmonic functions in the two-dimensional case and is a converse to Lemma \ref{lem:ref}. 

\begin{thm}
\label{thm:reh} 
Let $f$ be holomorphic on a domain $D$ in $\mathbb{C}$, where $h:=Re(f)$; then $h$ is harmonic on $D$.
\end{thm}

\begin{proof}
Let $f=h+ik$ then using the Cauchy-Riemann equations $h_x=k_y$ and $k_x=-h_y$ we have 
\begin{eqnarray}
\Delta h=h_{xx}+h_{yy}=k_{yx}-k_{xy}=k_{xy}-k_{xy}=0\notag
\end{eqnarray}
\end{proof}

Quickly looking at Example \ref{ex:haex} shows that the two functions given there, $p(x,y)=x$ and $f(x,y)=x^2-y^2$, are the real parts of the holomorphic maps $z\mapsto z$ and $z\mapsto z^2$. It is now relatively simple to create a large number of examples espcially bearing in mind that '$\Delta(.)$' is a linear differential operator, so that harmonic functions indeed form a vector space.

\subsection{Subharmonic Functions}

Before we can define subharmonic functions we need to recall semi-continuity and some of its basic implications.

\begin{defn} Let $(X,d)$ be a metric space, then $u:X\rightarrow[-\infty,\infty)$ is called \textbf{\index{upper-semicontinuous}upper-semicontinuous} (at $x$) if $\limsup_{y\rightarrow x} u(y)\leq u(x)$ for all $x\in X$. Similarly $v:X\rightarrow (-\infty,\infty]$ is called \textbf{\index{lower-semicontinuous}lower-semicontinuous} if $-v$ is upper-semicontinuous.  
\end{defn}

Note that we agree on the usual extended real number arithmetic, allowing functions to take one possible infinity and making sense of expressions like $0+\infty=\infty$ (see also \cite{1}).\\ 

The previous definition clearly implies that $u$ is upper-semicontinuos if and only if the set $\{x\in X:u(x)<\alpha\}$ is open for all $\alpha\in \mathbb{R}$. Using this fact we obtain an immediate, but very useful result about general upper-semicontinuous functions.

\begin{lem} If $u$ is upper-semicontinuous on a metric space $(X,d)$ and $K$ is a compact subset of $X$, then $u$ is bounded above on $K$ and attains its bound.
\label{lem:usc} 
\end{lem} 

\begin{proof}
The family of sets $\{x\in X:u(x)<n\}$ for $n\geq 1$ is an open cover of $K$ and by compactness we obatin a finite subcover, which implies that $u$ is bounded above. We can clearly set $M=\sup_{x\in K}u(x)$. Now take a sequence $\{x_{k}\}$ in $K$ such that $u(x_{k})\rightarrow M$. Since we can replace $\{x_{k}\}$ by a suitable subsequence if necessary, we may assume that $\{x_{k}\}\rightarrow x$ for some $x\in K$. Then using the definition of upper-semicontinuity we have
\begin{eqnarray}
M\geq u(x)\geq \limsup_{y\rightarrow x} u(y) \geq \lim_{k\rightarrow \infty} u(x_{k})=M
\end{eqnarray} 
\end{proof}

It should be noted that the definition and result above also hold for general topological spaces. Also there is an obvious analog of Lemma \ref{lem:usc} for lower-semicontinuous functions.

\begin{fact} If $v$ is lower-semicontinuous on a metric space $(X,d)$ and $K$ is a compact subset of $X$, then $u$ is bounded below on $K$ and attains its bound. 
\end{fact}

Furthermore it is a very important technical tool that we can in many cases approximate upper-semicontinuous functions by continuous functions, more precisely we have

\begin{lem}
\label{lem:uscapx}
Let $(X,d)$ be a metric space and let $u$ be upper-semicontinuous and bounded from above on X. Then there exist continuous functions $\phi_n:X\rightarrow\mathbb{R}$ which converge from above to $u$, i.e. $u\leq \phi_1\leq \phi_2\leq \ldots $ and $\lim_{n \rightarrow \infty} \phi_n=u$.
\end{lem}

\begin{proof} If $u\equiv -\infty$ take $\phi_n=-n$ and the result follows. If $u\not\equiv -\infty$ define $\phi_{n}(x)=\sup_{y\in X}\left(u(y)-nd(x,y)\right)$. The rest of the proof amounts to checking that this definition works; continuity holds since for any $n$ and $s,t\in X$ we have
\begin{eqnarray}
|\phi_n(s)-\phi_n(t)|\leq nd(s,t)\notag
\end{eqnarray} 
That the $\phi_n$ decrease and that $\lim_{n\rightarrow\infty}\phi_n \geq u$ follows from the construction. Furthermore for all $x\in X$ and $p>0$ we conclude that
\begin{eqnarray}
\phi_n\leq \max\left(\sup_{B(x,p)}u,\sup_{X} u-np\right)\notag\\
\Rightarrow \quad \lim_{n\rightarrow\infty} \phi_{n}(x)\leq \sup_{B(x,p)}u\notag
\end{eqnarray} 
Using upper-semicontinuity and taking the limit as $p\rightarrow 0$ gives the inequality $\lim_{n\rightarrow\infty}\phi_n\leq u$.
\end{proof}

Now we can proceed to define subharmonic functions, where we again focus on the two-dimensional case as already explained in the last chapter.

\begin{defn} Let $D$ be an open set in $\mathbb{C}$, then a function $u:D\rightarrow [-\infty,\infty)$ is called \textbf{\index{subharmonic}subharmonic} (at $w$) if it is upper-semicontinuous and there exists $p>0$ such that 
\begin{eqnarray}
u(w)\leq S^r_u(w)=\frac{1}{2\pi}\int_0^{2\pi} u(w+re^{it})dt \qquad \text{for $0\leq r < p$}
\end{eqnarray}
Furthermore $v:U\rightarrow (-\infty,\infty]$ is called \textbf{\index{superharmonic}superharmonic} if $-v$ is subharmonic. 
\end{defn}    

Notice the analogy to the definition of harmonic functions. Subharmonic functions satisfy the 'surface-mean' condition with an inequality - sometimes called 'submean'-inequality; therefore the definition can directly be generalized to $\mathbb{R}^n$. The definition also reveals that there is no preferred way to use subharmonic or superharmonic functions and there also seems to be no agreement in the literature, which terminology to use. We are going to work exclusively with subharmonic functions.

\begin{fact}If $h$ is subharmonic and superharmonic, then $h$ is harmonic. 
\end{fact} 

Of course, this obvious fact has to be understood locally as all definitions and statements about harmonicity. Clearly harmonic functions are subharmonic and the following theorem gives even more examples of subharmonic functions.

\begin{thm}
\label{thm:logsub} 
Let $f$ be holomorphic on an open set $U\subset \mathbb{C}$, then $\log|f(z)|=u(z)$ is subharmonic.
\end{thm} 

\begin{proof}
Since $\limsup_{w\rightarrow z}\log|f(w)|=\log|f(z)|$, upper-semicontinuity holds; it remains to check the second condition. But observe that if $w\in U$ and $u(w)=-\infty$ the 'sub-mean inequality' immediately holds, so assume that $-\infty<u(w)$. Now since in this case locally $\log|f|=Re(\log f)$ it follows from Theorem \ref{thm:reh} that $\log|f|$ is harmonic, where $\log f$ is an analytic branch of the logarithm. Therefore $\log|f|$ is subharmonic.   
\end{proof}

The following example is far more important than it might seem at first and illustrates the fact that subharmonic functions obey weaker conditions than harmonic functions, which is going to provide the necessary flexibility for several arguments we are going to develop later on.

\begin{ex} The first goal is to show that the following function is subharmonic on $\mathbb{C}$
\begin{eqnarray}
u(z)=-\sum_{n=1}^\infty 2^{-n}\log |z-2^{-n}|
\end{eqnarray}
It is easy to see that $u$ is upper-semicontinuous. We want to check the sub-mean inequality directly and therefore need to work out the following integral first
\begin{eqnarray}
\frac{1}{2\pi}\int_0^{2\pi} \log|re^{it}-w|dt=\begin {cases}\log|w| & \text{if $r\leq|w|$}\\ \log r & \text{if $r>|w|$} \end {cases}
\end{eqnarray}
Notice that the formula is immediate for $|w|>r$ since $\log|\mu - w|$ is a harmonic function of $\mu$ so by applying the '(surface)-mean' value property at $\mu=0$ the first part follows. The second part for $|w|<r$ is similar
\begin{eqnarray}
\frac{1}{2\pi}\int_0^{2\pi} \log|re^{it}-w|dt&=& \frac{1}{2\pi}\int_0^{2\pi} \log|r-we^{-it}|dt\notag\\
&=& \frac{1}{2\pi}\int_0^{2\pi} \log|r-\overline{w}e^{it}|dt=\log r\notag
\end{eqnarray}
It remains to consider the case $|w|=r$; using the previous results we get
\begin{eqnarray}
\frac{1}{2\pi}\int_0^{2\pi} \log|re^{it}-w|dt&=& \lim_{p\rightarrow r^-}\frac{1}{2\pi}\int_0^{2\pi}\notag \log|pe^{it}-w|dt=\log r
\end{eqnarray}
where the interchange of limits and integration is jusified by the dominated convergence theorem. We conclude the check of subharmonicity at $w$ by choosing $|w-2^{-n}|>r$ and $r\geq p>0$ 
\begin{eqnarray}
\frac{1}{2\pi}\int_0^{2\pi} u(w+pe^{it})dt&=& \frac{1}{2\pi}\int_0^{2\pi} \left(-\sum_{n=1}^\infty 2^{-n}\log|w+pe^{it}-2^{-n}|\right)dt\notag\\
&=& -\frac{1}{2\pi} \sum_{n=1}^\infty 2^{-n} \int_0^{2\pi} \log|w+pe^{it}-2^{-n}|dt=-\frac{1}{2\pi} \sum_{n=1}^\infty 2^{-n} \log|w-2^{-n}| \notag \\
&\geq& -\sum_{n=1}^\infty 2^{-n} \log|w-2^{-n}|=u(w)\notag
\end{eqnarray}
Now we can show that $u$ is not even continuous at $z=0$. Indeed 
\begin{eqnarray}
u(0)&=&- \sum_{n=1}^\infty 2^{-n} \log (2^{-n})= \sum_{n=1}^\infty n 2^{-n} \log 2\notag
\end{eqnarray} 
which converges by the root test as $\lim_{n\rightarrow \infty}1/2 \sqrt[n]{n}=1/2$. Now pick $z_n=2^{-n}$, then obviously $z_{n}\rightarrow 0$ as $n\rightarrow\infty$, but $\liminf u(z_n)=-\infty$ since for some $n$ we encounter $log|2^{-n}-2^{-n}|=log(0)=\infty$.
\end{ex}

This shows that subharmonic functions do not even need to be continuous in contrast to harmonic functions, which are automatically $C^{\infty}$. Despite the previous example, not all properties of harmonic functions are lost - some can partly (!) be recovered.

\begin{thm}[\index{Maximum Principle - subharmonic version}Maximum Principle - subharmonic version] 
\label{thm:maxsub}
Let $u$ be a subharmonic function on a domain $D$ in $\mathbb{C}$, then 
\begin{enumerate}
	\item if $u$ attains a global maximum on $D$ then $u$ is constant.
	\item if $\limsup_{z\rightarrow w}u(z)\leq 0$ for all $w\in\partial D$, then $u\leq 0$ on $D$.
\end{enumerate}
\end{thm}

\begin{proof}
Beginning with the first part we have that $u$ attains a global maximum $M$ on $D$. Then define
\begin{eqnarray}
A=\{z\in D:u(z)<M\}\qquad \text{and}\qquad B=\{z\in D:u(z)=M\}=D\backslash A\notag
\end{eqnarray}
Since $u$ is subharmonic it is upper-semicontinuous and therefore $A$ is open. By the submean inequality $B$ is open as well since if $u(z)=M$ there exist small circles around $z$ on which $u(z)=M$. By the usual argument that $A\cup B=D$, $A\cap B=\emptyset$, $D$ connected and $B\neq \emptyset$ we obtain $B=D$ and therefore $u$ is constant on $D$. For the second part of the theorem extend $u$ to $\partial D$ by defining
\begin{eqnarray}
u(w)=\limsup_{z\rightarrow w}u(z) \qquad \text{for $w\in \partial D$}\notag
\end{eqnarray}    
By construction $u$ is upper-semicontinuous on $\overline{D}$. Since $\overline{D}$ is compact Lemma \ref{lem:usc} on maxima of upper-semicontinuous functions applies to give $w\in \overline{D}$ such that $u(w)=\sup_{z\in\overline{D}}u(z)$. If $w\in\partial D$ then we have $u(z)\leq 0$ immediately. If $w\in D$ then the first part of the theorem implies that $u$ is constant on $D$ and therefore constant on $\overline{D}$ and it follows that $u\leq 0$ on $D$.     
\end{proof}

It is important to notice that in the maximum principle for harmonic functions it suffices to attain a local maximum whereas the maximum principle for subharmonic functions requires a global maximum. 

\begin{ex} To see the difference between the two maximum principles consider $u(z)=\max(Re(z),0)$. One verifies directly from the definition that $u$ is subharmonic. Observe that $u$ attains a local maximum (in the left-half plane) and also a global minimum (i.e. 0). Nevertheless, $u$ is not constant.  
\end{ex}

For completeness we remark that the definition of subharmonicity via the local sub-mean inequality can be extended.

\begin{fact}[\index{Global Submean Inequality}Global Submean Inequality, \cite{4}]
Let $U\subset \mathbb{C}$ be open. Then $u:U\rightarrow\mathbb{C}$ is subharmonic on $U$ if and only if for any $\overline{\Delta}(w,r)\subset U$ we have that 
\begin{eqnarray}
u(w)\leq\frac{1}{2\pi}\int_{0}^{2\pi}u(w+re^{it})dt
\end{eqnarray}
\end{fact}

Although we have seen many ways how to construct subharmonic functions, it is usually quite complicated to check whether a function is subharmonic. The following general criterion turns out to be useful in many cases.

\begin{thm}[\index{Criterion for Subharmonicity}Criterion for Subharmonicity]
\label{thm:subcrt}
Let $(X,M,\mu)$ be a measure space with $\mu$ being a finite measure. Consider an open set $U\subset \mathbb{C}$ and $v:U\times X\rightarrow [-\infty,\infty)$ such that
\begin{enumerate}
	\item $v$ is measurable on $U\times X$
	\item $z\mapsto v(z,w)$ is subharmonic on $U$ for all $w\in X$
	\item $z\mapsto \sup_{w\in X} v(z,w)$ is locally bounded above on $U$ 
\end{enumerate}
then it follows that $u(z)=\int_X v(z,w)d\mu(w)$ is subharmonic on U. 
\end{thm}

\begin{proof}
Note that if $u$ is subharmonic on each relatively compact subdomain $D$ of $U$ then $u$ is subharmonic on $U$ by considering a compact exhaustion. So let's consider only such a relatively compact subdomain $D\subset U$. By assumption \textit{3.} we have $\sup_{w\in D}v(z,w)<\infty$, so if $s_n\rightarrow s$ in $D$, then we can use Fatou's lemma and assumption \textit{2.} to get
\begin{eqnarray}
\limsup_{n\rightarrow \infty}u(s_n)\leq \int_X \limsup_{n\rightarrow \infty}v(s_n,w)d\mu(w)\leq \int_X v(s,w)d\mu(w)=u(s)\notag
\end{eqnarray}
Therefore $u$ is upper-semicontinuous. Then to prove subharmonicity we pick $\overline{\Delta(s,p)}\subset D$ and use Fubini's theorem (justified by the same reasoning as Fatou previously) to check the submean inequality
\begin{eqnarray}
\frac{1}{2\pi}\int_0^{2\pi} u(s+pe^{it})dt = \int_X \frac{1}{2\pi}\int_0^{2\pi} v(s+pe^{it},w)dtd\mu(w)\geq  \int_X  v(s,w)d\mu(w)=u(s)\notag
\end{eqnarray} 
\end{proof}

So far we have defined subharmonic functions via integral inequalities, which - although being well-defined - certainly might involve '$-\infty$' as a result of integration since upper-semicontinuity only imposes a bound from above and not from below. But the situation is better then one might guess at first

\begin{thm}[\index{Integrability}Integrability]
\label{thm:int} 
If $D$ is a domain in $\mathbb{C}$ and $u\not\equiv -\infty$ is subharmonic on $D$, then $u\in L^1_{loc}(D)$.
\end{thm}

\begin{proof}
First notice that $u$ is measurable by upper-semicontinuity. Let $K\subset D$ be compact. Since $K$ is compact finitely many discs of the form $\Delta(w_i,p_i)$ will cover $K$ for any open cover of discs, so it suffices to show that if for any $w\in D$, $\exists p>0$ such that 
\begin{eqnarray}
\int_{\Delta(w,p)}|u|dm<\infty\notag
\end{eqnarray}
Define $A=\{w\in D:\left(\exists p>0 : \int_{\Delta(w,p)}|u|dm<\infty\right)\}$ and $B=D\backslash A$. Since $D$ is a domain it is in particular connected. If we can show that $A$ and $B$ are both open it follows that $A\cup B=D$ and by connectivity either $B$ or $A$ are empty. But $u\not\equiv -\infty$ hence $A\neq\emptyset$ so that $A=D$ and the result follows.\\

So it remains to show that $A$ and $B$ are open. Let $w\in A$, so $\exists p>0$ such that $\int_{\Delta(w,p)}|u|dm<\infty$. Now we construct for any $w_0$ a disc inside $\Delta(w,p)$ belonging to $A$. So define $p_0=p-|w-w_0|$ so that $\Delta(w_0,p_0)\subset \Delta(w,p)$. Then indeed
\begin{eqnarray}
\int_{\Delta(w_0,p_0)}|u|dm\leq \int_{\Delta(w,p)}|u|dm<\infty\notag
\end{eqnarray}    
For the set $B$ we proceed in a more subtle way; pick $w\in B$, then there exists $p>0$ so that $\overline{\Delta(w,3p)}\subset D$ and $\int_{\Delta(w,p)}|u|dm=\infty$ as $w\in B$. Now for $w_0\in\Delta(w,p)$ define $p_0=p+|w-w_0|$ so that we have constructed a disc $\Delta(w_0,p_0)$ 'between' the two previous ones, i.e. $\Delta(w,p)\subseteq \Delta(w_0,p_0) \subseteq \Delta(w,3p)$. Notice that the upper-semicontinuity of $u$ on $\Delta(w_0,p_0)$ implies that 
\begin{eqnarray}
\int_{\Delta(w_0,p_0)}u \quad dm=-\infty\notag
\end{eqnarray}
But $u$ is subharmonic and therefore satisfies the global submean inequality and this implies for $0< r\leq p_0$
\begin{eqnarray}
u(w_0)&\leq& \frac{1}{2\pi}\int_0^{2\pi} u(w_0+re^{it})dt\notag\\
\Leftrightarrow \quad 2\pi r u(w_0) &\leq& r \int_0^{2\pi} u(w_0+re^{it})dt\notag
\end{eqnarray}
Then by integrating with respect to $r$ between $0$ and $p_0$ on both sides of the equation we get that
\begin{eqnarray}
\int_0^{p_0} 2\pi r u(w_0)dr &\leq& \int_0^{p_0}r \int_0^{2\pi} u(w_0+re^{it})dtdr\notag\\
\Rightarrow \quad  (p_0)^2 \pi u(w_0)  &\leq& \int_{\Delta(w_0,p_0)} u(w_0+re^{it})dm=-\infty\notag\\
\Rightarrow \quad u=-\infty \quad \text{on $\Delta(w,p)$ } \Rightarrow \quad \text{B is open}\notag  
\end{eqnarray}
\end{proof}

This shows that in general the integrals on compact sets for subharmonic functions are as nice as one could hope for; in addition it is also sensible to worry what happens if the circle of the submean inequality shrinks to a point.

\begin{lem} 
\label{lem:tech}
Let $u$ be subharmonic on $\Delta(0,p)$ with $u\not\equiv -\infty$ then for $0<r<p$ we have
\begin{enumerate}
	\item $\sup_{|z|=r}u(z)\geq \frac{1}{2\pi}\int_0^{2\pi} u(re^{it})dt\geq u(0)$
	\item $S_u^r(0):=S_u(r)=\frac{1}{2\pi}\int_0^{2\pi}u(re^{it})dt\rightarrow u(0) \quad \text{as $r\rightarrow 0$}$
	\item $\frac{1}{2\pi}\int_0^{2\pi} u(re^{it})dt$ is an increasing function of $r>0$ (on $\Delta(0,p)$)
\end{enumerate}
\end{lem} 

\begin{proof}
The first inequality follows from a direct ML-estimate
\begin{eqnarray}
S_u(r)=\frac{1}{2\pi}\int_0^{2\pi} u(re^{it})dt \leq \frac{1}{2\pi}\int_0^{2\pi} \sup_{|z|=r}u(z)dt=\sup_{|z|=r}u(z)\notag
\end{eqnarray}
and the second from the fact that $u$ is subharmonic using the definition. Continuing with \textit{2.} we note that by the upper-semicontinuity of $u$
\begin{eqnarray}
\limsup_{r\rightarrow 0} \left(\sup_{|z|=r}u(z)\right)\leq u(0)\notag
\end{eqnarray}
Now apply \textit{1.} to conclude the proof of the second part. For the third claim apply Theorem \ref{thm:subcrt} to see that $S_u(r)$ is subharmonic on $\Delta(0,p)$. Given $r_1,r_2\in [0,p)$ with $r_1<r_2$ apply the maximum principle for subharmonic functions to $S_u(r)$ on $\Delta(0,r_2)$ and the fact that $S_u(r)$ is a radial function to get the relation
\begin{eqnarray}
S_u(r_1)\leq \max_{\partial\Delta(0,r_2)}S_u(r)=S_u(r_2)\notag
\end{eqnarray}   
which exactly states that $S_u(r)$ is increasing.
\end{proof}

The previous technical lemma suffices to prove a standard result about the approximation of subharmonic functions via convolutions.

\begin{thm}[\index{Smoothing}Smoothing] 
\label{thm:smoothing}
Let $u$ be subharmonic on a domain $D$ in $\mathbb{C}$ such that $u\not\equiv -\infty$. Let $\chi:\mathbb{C}\rightarrow\mathbb{R}$ be an approximate identity.
For $r>0$ and $z\in\mathbb{C}$ define as usual (see also Definition \ref{defn:ai})
\begin{eqnarray}
\chi_{r}(z)&=&\frac{1}{r^2}\chi_{r}\left(\frac{z}{r}\right) \notag\\
D_{r}&=& \{ z \in D : d\left(z,\partial D\right)>r \} \notag
\end{eqnarray}
Then $u\ast \chi_{r}\in C^{\infty}(D_{r})$ and $u$ is subharmonic in $D_r$ for any $r>0$. Also $u\ast \chi_{r}\downarrow u$ on $D$ as $r\downarrow 0$. 
\end{thm}

\begin{proof}
By the integrability theorem (see \ref{thm:int}) and Fact \ref{fact:conv} it follows that $u\ast\chi_{r}\in C^{\infty}(D_{r})$ as claimed. To see that $u\ast\chi_{r}$ is subharmonic apply the criterion for subharmonicity (Theorem \ref{thm:subcrt}) for $(X,M,\mu)=(\mathbb{C}, \text{Borel }\sigma\text{-algebra}, \chi dm)$ and $v(z,w)=u(z-w)$.\\

For convergence first fix $c\in D$ and consider $r$ such that  $0<r<d(c,\partial D)$. Then by the definition of convolution and the fact that $\chi$ is a radial function we obtain
\begin{eqnarray}
u\ast \chi_r(c)=\int_0^{2\pi} \int_0^r u(c-se^{it})\frac{1}{r^{2}}\chi\left(\frac{s}{r}\right) s ds dt\notag
\end{eqnarray}
Now one can substitute $\sigma=s/r$ and $v(z)=u(c-z)$ and use Fubini's theorem to deduce
\begin{eqnarray}
u\ast \chi_r(c)&=&\int_0^{2\pi} \int_0^1 v(se^{it})\chi\left(\sigma\right) \sigma d\sigma dt\notag\\
&=& \int_0^{1} 2\pi \underbrace{\frac{1}{2\pi} \int_0^{2\pi} v(se^{it}) dt}_{:=C_v(s)} \chi\left(\sigma\right) \sigma d\sigma \notag
\end{eqnarray}
Then $C_v(s)$ is exactly the setup we need for the previous lemma (i.e. \ref{lem:tech}) and since $C_v(s)$ is increasing, we can use the monotone convergence theorem to get
\begin{eqnarray}
\lim_{r\downarrow 0} u\ast \chi_{r}(c) &=&  \lim_{r\downarrow 0} \int_0^{1} 2\pi \frac{1}{2\pi} \int_0^{2\pi}v(se^{it}) dt \chi\left(\sigma\right) \sigma d\sigma  \notag\\
&=& 2\pi \int_0^1 v(0) \chi(\sigma)\sigma d\sigma=u(c)\int_{\mathbb{C}}\chi dm=u(c)\notag
\end{eqnarray}   
where we have used the unit integral property of the approximate identity in the last step. Hence since $c$ was arbitrary we have $u\ast \chi_{r}\downarrow u$ on D.
\end{proof}

We haven shown that subharmonic functions can be approximated by smooth, subharmonic functions. This turns out to be a very powerful method as we shall see later in proving the Riesz decomposition theorem, thereby justifying the amount of work and technical tools (monotone convergence, Fubini, maximum principle, approximate identities, etc.) we have invested to develop this technique.

\subsection{Potentials}

Having developed harmonic and subharmonic functions the next goal is to relate them. The next definition fixes the main object of study in this section.

\begin{defn} [Logarithmic Potential] 
\label{defn:logpot}
Let $\mu$ be a finite Borel measure on $\mathbb{C}$ with compact support. The \textbf{\index{logarithmic potential}logarithmic potential} of $\mu$, denoted $p_{\mu}$ is a function from $\mathbb{C}$ to $[-\infty,\infty)$ defined as
\begin{eqnarray}
p_{\mu}(z)=\int_{\mathbb{C}}\log|w-z|d\mu(w)
\end{eqnarray} 
\end{defn} 
The definition seems to be unmotivated, when one encounters it first. The following example gives a motivation, why one might want study the logarithm in two dimensions in the context of potential theory.

\begin{ex}
\label{ex:fundsol}
 Suppose for this example that we are working in $\mathbb{R}^{n}$ and let us try to find a harmonic function on the annulus $B(0,r_1,r_2)=\{x\in \mathbb{R}^n:r_1< \|x\| <r_2\}$. First we assume that $u$ is radial and write $u_\ast(r)=u(x)$ for $|x|=r$. We know that $u$ must be differentiable if it exists and its Laplcian is given by a simple calculation as 
\begin{eqnarray}
\Delta u(x)=\frac{d^2u_\ast}{dr^2}(|x|)+\frac{n-1}{r}\frac{du_\ast}{dr}(|x|)
\end{eqnarray}
Hence starting from a partial differential equation - find $u$ such that $\Delta u=0$ - we have obtained an ordinary second order diffenretial equation for $u_\ast$. Note that the geometry of the domain plays a crucial role in this setup. In any case, the ODE can easily be solved to give
\begin{eqnarray}
u(x)=\begin{cases} A\log \frac{1}{|x|}+B & \text{if $n=2$}\\ \frac{1}{|x|^{n-2}}+B & \text{if $n\leq 3$} \end{cases}
\end{eqnarray} 
\end{ex} 

From the example we observe the special role of the logarithm of the absolute value in dimension 2. It turns out that the analogue for $n\geq 3$ for the potential is not as straightforward as all results in the previous sections.

\begin{defn} [Newtonian Potential]
\label{defn:newton}
Let $n\geq 3$ and consider $\mu$ being a finite Borel measure on $\mathbb{R}^n$ with compact support. The \textbf{\index{Newtonian potential}Newtonian potential} of $\mu$, denoted $p_{\mu}$ is a function from $\mathbb{R}^n$ to $[-\infty,\infty)$ defined as
\begin{eqnarray}
p_{\mu}(x)=-\int_{\mathbb{R}^n}\frac{1}{|w-x|^{n-2}}d\mu(w)
\end{eqnarray}  
\end{defn} 

The two definitions merit a few comments. First of all we define the Newtonian potential here only for illustration purposes and to make the reader aware that some differences between 2 and more dimensions can possibly occur. One major technical difficulty when dealing with potentials is unique to the planar case, namely the logarithm is of no definite sign in contrast to the function $u(x)=\frac{1}{|x|^{n-2}}$. \\

On the other hand, complex analysis simplies several proofs, so that in essence the same amount of work has to be done and many results are esentially the same even though the definitions of the two potentials are different. We are going to stay with the logarithmic potential on $\mathbb{C}$.\\

The reader might have already wondered, why we have introduced an additional sign instead of considering $\int_{\mathbb{C}}1/\log|w-z|d\mu(w)$. The next result shows that there is no fundamental difference except that we already agreed to work with subharmonic functions instead of superharmonic ones.

\begin{thm}
The logarithmic potential $p_\mu$ is subharmonic on $\mathbb{C}$.
\end{thm}

\begin{proof}
Again we want use the criterion for subharmonicity (Theorem \ref{thm:subcrt}). So let $S=supp(\mu)$, then we can regard $(K,\mu)$ as a measure space with $\mu$ restricted to $K$. Now define $v(z,w)=\log|z-w|$ and observe that it is immediate that $z\mapsto v(z,w)$ is subharmonic for fixed $w$ as $z\mapsto z-w$ is holomorphic (see Theorem \ref{thm:logsub}). Now the criterion applies to yield the result.  
\end{proof}

Hence without the extra sign we would have obtained a superharmonic function instead.\\

The next goal is to 'differentiate' the potential. Note that $p_\mu$ is clearly not differentiable in general. Therefore it seems sensible to briefly recall here a canonical construction to define a 'derivative' in a weaker sense.

\begin{defn} 
Let $U$ be an open subset of $\mathbb{R}^2$ (or $\mathbb{C}$ under identification). Denote by $\cal{D}(U)$ the space of \textbf{\index{test functions}test functions} $\phi:U\rightarrow \mathbb{R}$ such that $\phi\in C^\infty (U)$ and $\phi$ has compcat support contained in $U$.  
\end{defn}

\begin{defn} If $\phi_n,\phi\in \cal{D}(U)$, then $\phi_n\rightarrow\phi$ in $\cal{D}(U)$ if $\exists K\subset U$ such that $supp(\phi_n)\subset K$ for all $n$ and $D^{\alpha}\phi_n\rightarrow D^\alpha \phi$ uniformly for all multiindices $\alpha$.
\end{defn}

\begin{defn}
A function $T:\cal{D}(U)\rightarrow\mathbb{R}$ is called a \textbf{\index{distribution}distribution} if $T$ is linear and $T(\phi_n)\rightarrow T(\phi)$ whenever $\phi_n \rightarrow \phi$ in $\cal{D}(U)$. 
\end{defn}

Now notice that if we have an open set $U\subset \mathbb{R}^{2}$ and $f:U\rightarrow\mathbb{R}$ such that $f\in L^1_{loc}(U)$ then $f$ defines a distribution $T_f$ via
\begin{eqnarray}
T_f(\phi)=\int_U \phi(x)f(x)dm(x)
\end{eqnarray}
Hence one can view $f$ as a function taking points or as a distribution taking test functions as arguments. Since test functions are infinitely differentiable we can make sense of the following.

\begin{defn} 
Let $T$ be a distribution on $\cal{D}(U)$ then define for $x=(x_1,x_2)$
\begin{eqnarray}
\frac{\partial T}{\partial x_{j}}(\phi)=-T\left(\frac{\partial \phi}{\partial x_{j}}\right) \quad \text{for all $\phi\in\cal{D}(U)$}
\end{eqnarray}
which is called the \textbf{j-th partial derivative}. If $T=T_f$ for some function $f\in L^1_{loc}(U)$ then it is also called the \textbf{j-th \index{distributional derivative}distributional derivative} of $f$. 
\end{defn}

\begin{fact}[\cite{3}] 
The j-th distributional derivative is identical with the ordinary j-th derivative if $f$ is differentiable.
\end{fact}

Although we shall not need the previous fact it is good to know that the theory turns out to be a coherent extension if the functions are sufficiently smooth. Far more important is the following result about distributions.

\begin{thm}
\label{thm:disex}
If $T$ is a non-negative distribution (i.e. $T(\phi)\geq 0$, $\forall\phi\in\cal{D}(U)$) in the open set $U\subseteq \mathbb{R}^{n} $ then there exists a Radon measure $d\mu$ (i.e. non-negative, finite on compacts) in $U$ such that $T(\phi)=\int_U \phi(x) d\mu(x)$. 
\end{thm} 

Therefore we can associate a Radon measure to any distribution. On the one hand the proof of the theorem conveys, which results are necessary for all the following constructions, but on the other hand the proof is supposed to be well-known, therefore we content ourselves with an outline.

\begin{proof}(Sketch)
Consider an open exhaustion $U_m$ of $U$; then there exists $\phi_0\in\cal{D}(U)$ such that $0\leq \phi_0\leq 1$ on $U$ and $\phi_0=1$ on $\overline{U_m}$. Since for all $\phi\in\cal{D}(U_\text{m})$ we have $|\phi|\leq \| \phi\|_\infty \phi_0$ in $U$ and $T\geq 0$ it follows that $|T(\phi)|\leq \| \phi \|_\infty |T(\phi_0)|$. Hence $T$ is a bounded linear functional on $\cal{D}(U_\text{m})$ and using the Hahn-Banach theorem it extends to a linear functional on $C(\overline{U_m})$. Then the Riesz representation theorem applies to give a unique Radon measure $d\mu_m$ in $\overline{U_m}$ such that 
\begin{eqnarray}
T(\phi)=\int_{\overline{U_m}} \phi(x) d\mu(x)\notag
\end{eqnarray}
Now if $A$ is a Borel subset of $U$ and $A\subset U_m$ we define $d\mu(A)=d\mu_m(A)$. It is easy to check that this is well-defined and yields the claimed measure.
\end{proof}

In particular if $U\subseteq \mathbb{C}$ is a domain and we have $u:U\rightarrow [-\infty,\infty)$ such that $u$ is subharmonic and $u\not\equiv -\infty$ we have a very useful notion in the context of potential theory.

\begin{thm}
There exists a unique Radon measure denoted by $\Delta u$ such that 
\begin{eqnarray}
\int_U \phi \Delta u=\int_U u\Delta \phi dm \qquad \text{ for all $\phi\in C_c^\infty(U)=\cal{D}(U)$}
\end{eqnarray}
\end{thm}

\begin{proof} Obvious from Theorem \ref{thm:disex}.
\end{proof} 

The reason for developing the terminology is that given a potential $p_\mu$ we might hope to recover $\mu$ by considering the Laplacian of $p_\mu$ in some form as $-log|z|$ is the solution to the problem presented in Example \ref{ex:fundsol}.

\begin{thm}
\label{thm:prm}
Let $\mu$ be a finite Borel measure on $\mathbb{C}$ with compact support. Then $\Delta p_\mu=2\pi\mu$.  
\end{thm} 

This is a very helpful result allowing us to switch between potentials and their associated measures. For the proof we require one of the incarnations of Green's formula; we state the version we use here for convenience.

\begin{fact}[\index{Green's formula}Green's formula, \cite{5}]
\label{fact:gr1} 
Let $U$ be an open set with $C^{1}$ boundary in $\mathbb{R}^2$ and let $f$ and $g$ be in $C^{2}(U)$ then 
\begin{eqnarray}
\int_U f\Delta gdm^{(2)} = \int_U g\Delta fdm^{(2)}-\int_{\partial U} \left(f\frac{\partial g}{\partial n}- g\frac{\partial f}{\partial n}\right)d\sigma
\end{eqnarray}
where $\frac{\partial }{\partial n}$ denotes the derivative in the direction of the inner normal and $d\sigma$ denotes integration with respect to arc length.
\end{fact}

\begin{proof} (of Theorem \ref{thm:prm}) Let $\phi\in\cal{D}(\mathbb{C})$; first note that $\log|z|\in L^1_{loc}(\mathbb{C})$ by Theorem \ref{thm:logsub} and Theorem \ref{thm:int}. We also have that $\Delta \phi$ is bounded and has compact support, so we can justify Fubini's theorem in the following calculation
\begin{eqnarray}
\label{eq:prm}
\int_\mathbb{C}p_\mu \Delta dm &=& \int_\mathbb{C}\int_\mathbb{C}\log|z-w| d\mu(w) \Delta\phi(z) dm(z)\notag\\
&=& \int_\mathbb{C}\underbrace{\int_\mathbb{C}\log|z-w| \Delta\phi(z) dm(z)}_{:=A}d\mu(w) 
\end{eqnarray}
Before we can apply Green's formula we point out the following auxillary calulation for the normal derivative on the circle for $z=x+iy=se^{it}$
\begin{eqnarray}
\label{eq:grn}
\frac{\partial}{\partial n}\log|z|=-\left(\frac{x}{\sqrt{x^2+y^2}},\frac{y}{\sqrt{x^2+y^2}}\right) \cdot \nabla \left(\log \sqrt{x^2+y^2}\right)=-\frac{1}{s}\left(\frac{x^2}{x^2+y^2}+\frac{y^2}{x^2+y^2}\right)=-\frac{1}{s}
\end{eqnarray}
Being well-prepared to apply Green's formula with equation (\ref{eq:grn}) and using that $\log(z)$ is harmonic outside of a disc around $0$ we obtain 
\begin{eqnarray}
A=\int_\mathbb{C}\log|z-w| \Delta\phi(z) dm(z)&=& \lim_{s\rightarrow 0} \int_{|z-w|>s} \log|z-w| \Delta\phi(z) dm(z) \notag\\
&=& \lim_{s\rightarrow 0} \int_0^{2\pi} \left(\phi(w+se^{it}) -s\log s \frac{\partial \phi}{\partial r}(w+re^{it})_{r=s} \right)dt=2\pi \phi(w)\notag
\end{eqnarray}
Hence we can use this in equation (\ref{eq:prm}) to conclude that
\begin{eqnarray}
\int_\mathbb{C} p_\mu \Delta \phi dm = \int_\mathbb{C} \phi 2\pi d\mu \notag
\end{eqnarray}
and since $\phi$ was arbitrary the result follows.
\end{proof}

Having developed the distributional Laplacian, the following result is essential.

\begin{lem}[\index{Weyl's Lemma}Weyl's Lemma]
If $u$ and $v$ are subharmonic on a domain $D\subset \mathbb{C}$ then if $u,v \not\equiv -\infty$ and $\Delta u=\Delta v$ then $u=v+h$ where $h$ is harmonic on $D$.  
\end{lem}

\begin{proof}
Let $\chi_r(z)=\frac{1}{r^2}\chi (\frac{z}{r})$ be an approximate identity and set $D_r=\{z\in D:d(z,\partial D)>r\}$ as used in Theorem \ref{thm:smoothing}. Then $u\ast \chi_r\in C^\infty(D_r)$ for any $z\in\ D_r$. Furthermore observe that
\begin{eqnarray}
\Delta(u\ast\chi_r)&=&\int_\mathbb{C} u(w)\Delta_z \chi_r (z-w) dm(w)=\int_\mathbb{C} u(w)\Delta_w \chi_r (z-w) dm(w)=\int_\mathbb{C} \chi_r (z-w) \Delta u\notag 
\end{eqnarray}
Repeating the same calculation with $v$ instead of $u$ we obtain $\Delta(u\ast\chi_r)=\Delta(v\ast\chi_r)$. Note that the smoothness of the convolution with an approximate identity is crucial as we can deduce the existence of a harmonic function $h_r$ such that $u\ast\chi_r=v\ast\chi_r+h_r$ for all $r>0$. Now apply Theorem \ref{thm:smoothing} to $h_r$ and $-h_r$ to conclude that $h_r\ast\chi_s=h_r$ on $D_{r+s}$ for $s>0$. Therefore
\begin{eqnarray}
h_r=h_r\ast \chi_r=(u-v)\ast\chi_r \ast\chi_s=h_s\ast\chi_r=h_s\notag
\end{eqnarray}
Hence the identity principle for harmonic functions (see Theorem \ref{thm:idha}) implies that there exists a unique harmonic function $h$ such that $u\ast\chi_r=v\ast\chi_r+h$. By letting $r$ tend to zero and using (again) Theorem \ref{thm:smoothing}, we obtain $u=v+h$ as claimed.  
\end{proof}

Now we are ready for the fundamental result, which establishes a link between the main objects of study we have considered so far.

\begin{thm}[\index{Riesz Decomposition}Riesz Decomposition]
Let $D$ be a domain in $\mathbb{C}$ and let $u:D\rightarrow \mathbb{R}$ be a subharmonic function, which is not identically equal to $-\infty$. If $K\subseteq D$ is an open and relatively compact set then there exists a harmonic function $h$ on $K$ such that
\begin{eqnarray}
u=p_{\mu}+h \qquad on \quad K
\end{eqnarray} 
where $2\pi\mu=\Delta u|_{K}$
\end{thm}

\begin{proof}
The first step is to define $\mu=\frac{1}{2\pi}\Delta u|_{K}$ and to compute $\Delta p_\mu=2\pi\mu=\Delta u$ by Theorem \ref{thm:prm}. Now apply Weyl's Lemma to each component of $D$ to obtain $u=p_\mu+h$. 
\end{proof}

Note that the proof turns out to be deceptively simple, but we clearly have already completed the most difficult steps in the proof of the smoothing theorem and by constructing the generalized Laplacian and deriving Weyl's lemma. Furthermore it should be noted that the Riesz decomposition links subharmonic functions closely to potentials and harmonic functions and once we can find out more about each piece of the decomposition one expects to derive information about the other classes of functions involved.

\section{Part II - Applications}

The general theory as presented in the first part turns out to have myriad extensions and an excellent reason to study these is their value for applications in a diverse set of different branches of mathematics. Using several examples we try to develop the general theory further motivated by particular applications.\\

We do not attempt in this part - in contrast to the proof of the Riesz decomposition theorem - to develop the complete details of all proofs, but we will assume some technical propositions stating them as facts together with a reference, where the proof can be found.

\subsection{Perron's Method and the Dirichlet Problem}

No work claiming to treat applications involving subharmonic functions would be complete without the Dirichlet Problem. 

\begin{defn}[\index{Dirichlet Problem}Dirichlet Problem]
\label{defn:dirch}
Let $U$ be an open, connected subset of $\mathbb{R}^n$ and let $f:\partial U \rightarrow \mathbb{R}$ be a bounded, continuous function. The problem is to find $u:\overline{U}\rightarrow \mathbb{R}$ such that $u\in C^{2}(U)$ and $u$ is continuous on $\partial U$ satisfying
\begin{eqnarray}
\Delta u&=&0 \quad \text{on $U$}\notag\\
u&=&f \quad \text{on $\partial U$}
\end{eqnarray}
\end{defn}

The problem expresses the physical situation of prescribing a temperature distribution $f$ at the boundary of $U$ and the solution $u$ will give the equilibrium heat distribution in the interior. The main questions one must ask from a mathematical viewpoint are\\

\textbf{Does there exist a solution to the Dirichlet problem and if so, is it unique?}\\

We are now mainly going to follow \cite{5} for the case, when $U$ is a disc and \cite{8} for the question of existence in the plane using an elegant treatment of the main technique of the proof from \cite{4}. The part dealing with uniqueness is taken from \cite{7}.\\

To begin with the easy part, the question of uniqueness has a straightforward answer for bounded domains.

\begin{thm}
Let $U$ be a bounded domain of $\mathbb{R}^n$, if a solution to the Dirichlet Problem (see \ref{defn:dirch}) exists then it is unique.
\end{thm}

\begin{proof} 
Let $u_1$ and $u_2$ be two solutions of the Dirchlet Problem. Suppose there exists $z_0\in U$ such that $u_1(z_0)\neq u_2(z_0)$. Assume without loss of generality that $u_1(z_0)> u_2(z_0)$. Then $u_1-u_2$ extends continuously to $\partial U$ with $u_1-u_2(z)=0$ for $z\in\partial U$. Since $u_1-u_2$ is harmonic on $U$ and $u_1-u_2>0$ on $U$, compactness of $\overline{U}$ implies that $u_1-u_2$ attains its maximum on $U$. Now by the maximum principle (Theorem \ref{thm:maxpri2} respectively Fact \ref{thm:maxpri3}) $u_1-u_2=0$, but this contradicts $u_1>u_2$, so by reversing the roles of $u_1$ and $u_2$ we conclude that we must have $u_1=u_2$.   
\end{proof}

Note that in general for unbounded domains uniqueness fails if we do not prescribe a condition at infinity. It is a technical computation to verify this e.g. for a half-space (see \cite{7}).\\

Indeed, solutions can be forced to be unique by prescribing boundary conditions at infinity, where we remind the reader that we have agreed to take boundaries in the plane with respect to $\mathbb{C}_\infty$.

\begin{fact}[\cite{4}]
If $U$ is a proper subdomain of $\mathbb{C}$, then if a solution to the Dirichlet Problem exists, then it is unique.
\end{fact} 

Making infinity an undistinguished point from $\mathbb{C}$ via the Riemann sphere $\mathbb{C}_\infty$ basically corresponds to boundary conditions at infinity for arbitrary $\mathbb{R}^n$. Although this point is quite subtle we want to emphasize it here as it marks another difference in the style of presentation in the literature between the planar and general case of $\mathbb{R}^n$ in potential theory, i.e. some authors work in two dimensions and agree taking closures with respect to $\mathbb{C}_\infty$ - like we did so far for $n=2$ - or work directly with $\mathbb{C}_\infty$, while general treatments for $\mathbb{R}^n$ usually distinguish between bounded and unbounded domains.\\

Having dealt with uniqueness, the question of existence requires far more thought as the following classical example shows.

\begin{ex}
Let $U=\Delta \backslash \{0\}\subset \mathbb{C}$ (respectively $\mathbb{R}^2$) be the punctured unit disc. Define $f(z)=1$ for $z=0$ and $f(z)=0$ for $|z|=1$. Suppose $u$ is a solution to Dirichlet problem on $U$. Observe that by the maximum principle $\limsup_{z\rightarrow \gamma} u(z)\leq f(\gamma)$ for $\gamma\in \partial U$. Hence the estimate $u(z)\leq \epsilon \log\left(\frac{1}{|z|}\right)$ holds for all $\epsilon>0$ for $|z|=1$ or $z$ close to $0$. But the maximum principle for harmonic functions implies that the estimate holds in $U$. Now let $\epsilon\rightarrow 0$ to obtain $u(z)\leq 0$ on $U$, but $u$ extends continuously to the boundary contradicting that $f(0)=1$.  
\end{ex}

Clearly the geometry of the domain and the point $0$ play a special role in preventing a solution to exist; this is confirmed by the fact that we can derive an explcit solution for the standard disc.

\begin{thm}
\label{thm:pifdisc}
If $u(z)$ is harmonic on $U=\Delta(0,R)=\Delta_R$ and continuous on $\partial \Delta_R$ then for all $z\in\Delta_R$
\begin{eqnarray}
u(z)= \frac{1}{2\pi} \int_0^{2\pi} \frac{R^2-|z|^2}{|Re^{is}-z|^2}u(Re^{is})ds
\end{eqnarray}
\end{thm}

\begin{proof}
Assume first that $u$ is also harmonic on $\overline{\Delta}_R$. Since $u$ is harmonic and $\overline{\Delta}_R$ is simply-connected by lemma $\ref{lem:ref}$ there exists a holomorphic function $g$ on $\overline{\Delta}_R$ such that $Re(g)=u$. Now extend $g$ analytically to a domain containing $\overline{\Delta}_R$. Then Cauchy's integral formula implies for $z\in\Delta_R$ 
\begin{eqnarray}
\label{eq:pois1}
g(z)=\frac{1}{2\pi i}\int_{|w|=R}\frac{g(w)}{w-z}dw=\frac{1}{2\pi}\int_0^{2\pi} \frac{g(w) w}{w-z}d\phi \quad (w=Re^{i\phi})
\end{eqnarray}
But by Cauchy's theorem with $|w_0|>R$ we obtain
\begin{eqnarray}
\label{eq:pois2}
\frac{1}{2\pi}\int_0^{2\pi} \frac{w g(w)}{w-w_0}d\phi=0
\end{eqnarray}
Now pick $w_0=R^2/\overline{z}$ and observe that $|w_0|>R$ since $z\in \Delta_R$ so combining the result of (\ref{eq:pois1}) and (\ref{eq:pois2}) it follows that
\begin{eqnarray}
\label{eq:pois3}
g(z)=\frac{1}{2\pi}\int_0^{2\pi} g(w) \left(\frac{w}{w-z}-\frac{w}{w-R^2/\overline{z}}\right)d\phi=\frac{1}{2\pi} \int_0^{2\pi} g(w)\frac{R^2-|z|^2}{|Re^{i\phi}-z|^2}d\phi
\end{eqnarray}
where again $w=Re^{i\phi}$. Then on taking the real part of both sides of (\ref{eq:pois3}) we have
\begin{eqnarray}
u(z)=\frac{1}{2\pi} \int_0^{2\pi} u\left(Re^{i\phi}\right)\frac{R^2-|z|^2}{|Re^{i\phi}-z|^2}d\phi\notag
\end{eqnarray}
Hence we have shown the case when $u$ is harmonic on $\overline{\Delta}_R$. If it is only harmonic on $\Delta_R$ apply the same argument to $u_r(z)=u(rz)$ for $r<1$. Since $u$ is uniformly continuous on $\overline{\Delta_R}$ interchanging the limit $r\rightarrow 1$ and integration is justified to obtain the claimed result.
\end{proof}

\begin{defn} 
The function $P(w,z)=\frac{|w|^2-|z|^2}{|w-z|^2}$ is called the \textbf{\index{Poisson kernel}Poisson kernel}.
\end{defn}

Note that we have an even simpler form of the Poisson kernel if we set $z=re^{i\theta}$ and $w=Re^{i\phi}$ to obtain
\begin{eqnarray}
P(w,z)=Re\left( \frac{w+z}{w-z} \right)=\frac{R^2-r^2}{R^2+r^2-2Rr\cos(\phi-\theta)}
\end{eqnarray} 

\begin{thm}
\label{thm:posintold}
If $g$ is continuous on $|w|=R$ and the Poisson kernel $P(w,z)$ is given as above then
\begin{eqnarray}
\label{eq:cand}
\frac{1}{2\pi}\int_0^{2\pi} P(w,z)g(w)d\phi \qquad (w=Re^{i\phi})
\end{eqnarray}
is harmonic on $\Delta_R$. 
\end{thm}

\begin{proof}
The only step required to prove the result is to observe that 
\begin{eqnarray}
\frac{1}{2\pi}\int_0^{2\pi} P(w,z)g(w)d\phi=Re\left(\frac{1}{2\pi}\int_0^{2\pi} \frac{w+z}{w-z}g(w)d\phi\right)\notag
\end{eqnarray}
and as real parts of holopmorphic functions are harmonic (see Theorem \ref{thm:reh}) the result follows.
\end{proof}

Hence we a constructed the candidate function in equation (\ref{eq:cand}), which is supposed to solve the Dirichlet Problem. It turns out that it indeed extends continuously to the boundary and hence is the required solution. The proof of this fact is technical and since we shall return to the problem of extending a solution to the boundary later, we omit it here. \\

Clearly we can shift the disc to some new center, say $z_0$, and all our previous calculations apply almost directly; the work can now be summarized.

\begin{thm} [\index{Poisson integral formula}Poisson integral formula]
\label{thm:posint}
If $D=\Delta(z_0,R)$ and $f:\partial D\rightarrow \mathbb{R}$ is continuous then the \textbf{\index{Poisson integral}Poisson integral} $P_D^f$ defined for $r<R$ and $\theta\in[0,2\pi)$ and given by 
\begin{eqnarray}
P_D^f(z_0+re^{i\theta})=\frac{1}{2\pi}\int_0^{2\pi} f(z_0+Re^{i\phi})\frac{R^2-r^2}{R^2+r^2-2Rr\cos(\phi-\theta)}d\phi
\end{eqnarray}
solves the Dirichlet Problem on $D$, i.e. $P_D^f$ is harmonic on $D$ and $\lim_{z\rightarrow \zeta } P_D^f(z)=f(\zeta)$ for all $\zeta\in \partial D$.
\end{thm} 

Although the clever use of some elementary complex analysis gave an explicit formula for the disc, even for plane domains one cannot expect this for arbitrary geometry. In the following we are going to present the main fundamental method of demonstrating existence of solutions for the Dirchlet Problem, the method can be adapted to work for $\mathbb{R}^n$ (see e.g. \cite{3} or \cite{7}), but we content ourselves with $\mathbb{C}$.\\

The first goal is to construct a harmonic function on a given domain D. 

\begin{defn}
\label{defn:perfam}
Let $D$ be a proper subdomain of $\mathbb{C}$ and let $f:\partial D\rightarrow \mathbb{R}$, we define the \textbf{\index{Perron family of subsolutions}Perron family of subsolutions} corresponding to $f$, denoted by $\cal{H}_D^{\text{$f$}}$ to be all subharmonic functions $u$ such that for $z\in D$
\begin{eqnarray}
\limsup_{z\rightarrow \zeta}u(z)\leq f(\zeta)\qquad \forall\zeta\in\partial D
\end{eqnarray}
\end{defn}

\begin{defn} 
\label{defn:persol}
We define the \textbf{\index{Perron solution}Perron solution} $H_D^f$ as the upper envelope of the Perron family of subsolutions. Namely we set $H_D^f(z)=\sup\{ u(z):u \in \cal{H}_D^{\text{$f$}} \}$.
\end{defn}

One motivation for the two definitions is that if there is a solution $u(z)$ harmonic on $D$ and $\lim_{z\rightarrow \zeta} u(z)=f(\zeta)$ then $u$ must be equal to the Perron solution. Indeed, if $u$ is solution it is certainly subharmonic and by definition we have $u(z)\leq H_D^f$. Let $v\in \cal{H}_D^{\text{$f$}}$; since for $w\in \partial D$ we have $\limsup_{z\rightarrow w} (v(z)-u(z))\leq 0$ the maximum principle allows us to conclude that $v(z)\leq u(z)$ on $D$, which implies $u(z)\geq H_D^f$.\\

Although we have the right definition we must show that a harmonic function actually exists (!) on a given domain. The following result is the major step.

\begin{thm}
\label{thm:perha}
Let $D$ be a proper subdomain of $\mathbb{C}$ and let $f:\partial D\rightarrow \mathbb{R}$ be continuous and bounded then the Perron solution $H_D^f$ is harmonic on $D$.
\end{thm}

Before we can carry out the proof we state two results, which are important in their own right, but require proofs, which do not provide new additional insight.

\begin{fact}[\index{Poisson modification}Poisson modification, \cite{4}] 
\label{thm:posmod}
If $D$ is a domain in $\mathbb{C}$, $\Delta$ is an open disc such that $\overline{\Delta}\subset D$ and $u\not\equiv -\infty$ is subharmonic on $D$ we define
\begin{eqnarray}
\tilde{u}(z)= \begin{cases} P^u_D & \text{on $\Delta$}\\
                            u & \text{on $D\backslash \Delta$}\end{cases}
\end{eqnarray}
Then $\tilde{u}$ is subharmonic on $D$, harmonic on $\Delta$ and satisfies $\tilde{u}\geq u$ on $D$. 
\end{fact}

One key idea is hidden in the proof of this fact, namely that $\tilde{u}$ is harmonic on $\Delta$. But this follows immediately from Theorem \ref{thm:posint} (respectively Theorem \ref{thm:posintold}), which is a consequence of the analysis of the case, when the domain is a disc.

\begin{fact}[\index{Harnack's Theorem}Harnack's Theorem, \cite{8}]
If $(u_n)_{n\geq 1}$ is a family of harmonic functions on $\mathbb{C}$ such that $h_1\leq h_2\leq \ldots$ on $D\subseteq \mathbb{C}$, where $D$ is a domain, then either $u_n$ converges locally uniformly to $\infty$ or $u_n\rightarrow u$ locally uniformly on $D$, where $u$ is harmonic. 
\end{fact}

Now we have all the necessary tools to start with the main existence result and prove that the Perron solution must be harmonic.

\begin{proof}(of Theorem \ref{thm:perha})
First we make an observation about the bounds on $H^f_D$. If we define $M=\sup_{\partial D}|f|$ then we have $-M\in\cal{H}_D^{\text{$f$}}$ and so $-M\leq H^f_D$. Also given $u\in \cal{H}_D^{\text{$f$}}$, the maximum principle (for subharmonic functions) gives that $u\leq M$ on $D$. Hence we must have $H^f_D\leq M$.\\

Note that if we show that $H^f_D$ is harmonic on each open disc $\Delta$ with $\overline{\Delta}\subset D$ then it follows that $H_D^f$ harmonic on $D$. So pick an arbitrary such disc $\Delta$ and any $w_0\in\Delta$. From the definition of Perron subsolutions we obtain a sequence $(u_n)_{n \geq 1}$ in $\cal{H}_D^{\text{$f$}}$ such that $u_n(w_0)\rightarrow H^f_D(w_0)$. Since we can replace $u_n$ with $\max(u_1,u_2,\ldots,u_n)$, where the maximum of two subharmonic functions is obviously subharmonic, we may assume without loss of generality that $u_1\leq u_2\leq \ldots$ ond $D$.\\

Now use the Poisson modification (see Fact \ref{thm:posmod}) of $u_n$, denoted by $\tilde{u}_n$ and note that the Poisson modifications still obey $\tilde{u}_1\leq \tilde{u}_2\leq \ldots$ on $D$. Now we intend to verify three properties for $\tilde{u}=\lim_{n\rightarrow \infty}\tilde{u}_n$, namely
\begin{enumerate}
	\item $\tilde{u}\leq H_D^f$ on $D$
	\item $\tilde{u}(w_0)=H_D^f(w_0)$
	\item $\tilde{u}$ is harmonic on $\Delta$ 
\end{enumerate}
We know that $\tilde{u}$ is subharmonic on $D$ from the Poisson modification, but also  
\begin{eqnarray}
\limsup_{z\rightarrow\zeta} \tilde{u}_n(z)=\limsup_{z\rightarrow\zeta} u_n(z)\leq f(\zeta)\notag
\end{eqnarray}  
for all $\zeta\in\partial D$; therefore $\tilde{u}_n\in\cal{H}_D^{\text{$f$}}$ for all $n$. Hence $\tilde{u}_n\leq H_D^f$ for all $n$ simply by definition of $H^f_D$ as an upper envelope and \textit{1.} follows. Since $\tilde{u}_n\geq u_n$ by the Poisson modification properties 
\begin{eqnarray}
\tilde{u}(w_0)=\lim_{n\rightarrow\infty}\tilde{u}_n(w_0)\leq \lim_{n\rightarrow\infty}{u}_n(w_0)=H^f_D(w_0)\notag
\end{eqnarray} 
Now combining this with \textit{1.}, part \textit{2.} is verified as well. Note that each $\tilde{u}_n$ is harmonic on $\Delta$ (the key insight from Poisson modification) and Harnack's theorem applies so that $\tilde{u}$ is harmonic on $\Delta$ since it cannot be identically equal to infinity as $\tilde{u}\leq H^f_D\leq M$ by the initial observation. Now if we would have $\tilde{u}=H^f_D$ on $\Delta$ the overall result would follow. So we proceed to show this.\\

Choose any $w\in\Delta$ and a sequence $(v_n)_{n\geq 1}$ in $\cal{H}_D^{\text{$f$}}$ such that $v_n(w)\rightarrow H^f_D(w)$. Again using the fact that we can replace $v_n$ by $\max(u_1,u_2,\ldots,u_n,v_1,\ldots,v_n)$ we may without loss of generality assume that $v_1\leq v_2\leq \ldots$ and $u_n\leq v_n$ for all $n$ on $D$. We use the notation $\tilde{v}_n$ for the Poisson modification of $v_n$. Clearly we have $v_n\uparrow v$ and the previously proven three properties imply
\begin{enumerate}
	\item $\tilde{v}\leq H_D^f$ on $D$
	\item $\tilde{v}(w)=H_D^f(w)$
	\item $\tilde{v}$ is harmonic on $\Delta$ 
\end{enumerate}
Using \textit{1.} we have $\tilde{v}(w_0)\leq H^f_D(w_0)=\tilde{u}(w_0)$. But since $\tilde{v}_n\geq \tilde{u}_n$ by construction, it follows that $\tilde{u}-\tilde{v}$ is harmonic on $\Delta$ and attains a maximum value of $0$ at $w_0$. The maximum principle (for harmonic functions) implies $\tilde{u}-\tilde{v}=0$ on $\Delta$. Therefore 
\begin{eqnarray}
H^f_D(w)=\tilde{v}(w)=\tilde{u} \quad \Rightarrow \quad \tilde{u}=H^f_D\notag
\end{eqnarray}
and the proof is finished.
\end{proof}

Being confident, how to construct in any situation the solution of the Dirichlet Problem, it remains to figure out, which conditions we must impose on the boundary of $D$ so that we can avoid problems as with the isolated point $0$ in the example of the punctured disc, where no solution exists.

\begin{defn}
Let $D$ be a proper subdomain of $\mathbb{C}$. A point $\zeta_0\in\partial D$ is called \textbf{\index{regular}regular} if for $z\in D$ we have $\lim_{z\rightarrow \zeta_0} H^f_D(z)=f(\zeta_0)$ for all bounded continuous functions $f:\partial D\rightarrow \mathbb{R}$. 
\end{defn}

Just defining points to admit a solution is not sufficient until we find a criterion to decide whether a point is regular or not. The following concept turns out to be correct.

\begin{defn} A \textbf{\index{subharmonic barrier}subharmonic barrier} at $\zeta_0\in \partial D$ is a subharmonic function $b(z)$ such that for some $\delta>0$ and $b:D\cap \Delta(\zeta_0,\delta)\rightarrow\mathbb{R}$ the following properties hold
\begin{description}
	\item[(a)] $b(z)<0$ for all $z\in D\cap \Delta(\zeta_0,\delta)$
	\item[(b)] $\lim_{z\rightarrow \zeta_0}b(z)= 0$ for $z\in D$
	\item[(c)] $\limsup_{z\rightarrow \zeta} b(z)<0$ as $z\in D$ and $\zeta\in\partial D$ with $0<|\zeta-\zeta_0|<\delta$  \end{description}
\end{defn}

Before we can prove the main relation between barriers and regular points we need two preliminary results. The first one allows us to relax conditions (a) and (c) of the previous definition and to extend the barrier to a larger domain.

\begin{lem}
\label{lem:regb1}
If $\zeta_0\in\partial D$ admits a barrier $b$, then $b$ can be extended to $D$ such that $\limsup_{z\rightarrow \zeta}b(z)<0$ for any $\zeta\in\partial D\backslash\{\zeta_0\}$ and $b(z)<0$ for all $z\in D$.
\end{lem}

\begin{proof} 
By using (a) and (c) we can choose $\epsilon>0$ such that $b(z)\leq -2\epsilon$ if $z\in D$ and $|z-\zeta_0|=\delta/2$. Now define $u(z)$ on $D$ by 
\begin{eqnarray}
u(z)=\begin{cases} \max(b(z),-\epsilon) & \text{if $z\in D\cap \{|z-\zeta_0|<\delta/2\}$}\\
-\epsilon & \text{otherwise} \end{cases} \notag
\end{eqnarray} 
It is easy to check that $u(z)$ is indeed a subharmonic barrier at $\zeta_0$ satisfying the required conditions. 
\end{proof}

\begin{lem}
\label{lem:regb2}
Let $D$ be a proper subdomain of $\mathbb{C}$ and let $f_1,f_2:\partial D\rightarrow \mathbb{R}$ be continuous, then we have for the Perron solutions $H_D^{f_1}+H_D^{f_2}\leq H_D^{f_1+f_2}$.
\end{lem}

\begin{proof}
If $u_1\in \cal{H}_D^{\text{$f_1$}}$ and $u_2\in \cal{H}_D^{\text{$f_2$}}$ then by definition of Perron subsolutions $u_1+u_2\in \cal{H}_D^{\text{$f_1+f_2$}}$, thus we have $u_1+u_2\leq H_D^{f_1+f_2}$. Now taking suprema over $u_1$ and $u_2$ gives the result.  
\end{proof}

\begin{thm}
Let $D$ be a proper subdomain of $\mathbb{C}$. If there exists a subharmonic barrier at $\zeta_0$, then $\zeta_0$ is a regular boundary point of $D$.
\end{thm}

\begin{proof}
The key idea is that the subharmonic barrier imposes a constraint on the convergence behaviour of the Perron solution at the boundary. Let $f:\partial D\rightarrow\mathbb{R}$ be continuous and bounded; notice that we can without loss of generality assume that $\zeta_0\in\mathbb{C}$, $f(\zeta_0)=0$ and $|f(\zeta)|<1$ for $\zeta\in\partial D$. Now fix $\epsilon >0$ and find $\delta>0$ such that $|f(\zeta)|<\epsilon$ for $\zeta \in \partial D$ and $|\zeta-\zeta_0|<\delta$.\\

By the assumption and the preliminary Lemma \ref{lem:regb1} there exists a subharmonic barrier $b(z)$ at $\zeta_0$ defined on $D$. Choose $p>0$ such that $b(z)\leq -p$ for $|z-\zeta_0|\geq \delta$, which is possible by the definition of a barrier. Then define $u(z)=\left(b(z)/p-\epsilon\right)$, which is again subharmonic since positive scaling and adding constants preserves subharmonicity. By the definition of $u$ we have that
\begin{eqnarray}
u(z)\leq -\epsilon \quad \text{on $D$} \qquad \text{and} \qquad \limsup_{z\rightarrow\zeta} u(z)\leq -1 \quad \text{for $\zeta\in\partial D$ and $|\zeta-\zeta_0|\geq\delta$}\notag
\end{eqnarray} 
Therefore $\limsup_{z\rightarrow \zeta}u(z)\leq f(\zeta)$ for $\zeta\in\partial D$ as $|f(\zeta)|<1$. Note that $u\in \cal{H}_D^{\text{$f$}}$ and thus $u(z)\leq H^f_D(z)$.\\

Furthermore since $\lim_{z\rightarrow\zeta_0}b(z)=0$ we have that $\lim_{z\rightarrow\zeta_0}u(z)=-\epsilon$ and it follows that $-2\epsilon \leq H^f_D(z)$ for $\zeta$ in some small disc around $\zeta_0$. Now apply the same argument to $-f$ instead of $f$ to get $-2\epsilon \leq H^{-f}_D(z)$ for $\zeta$ in some small disc around $\zeta_0$. Now by the second preliminary Lemma \ref{lem:regb2} we have that $H^f_D(z) + H^{-f}_D(z)\leq 0$ and therefore $H^f_D(z)\leq H^{-f}_D(z)\leq 2\epsilon$ again under the restriction that $\zeta$ is in some sufficiently small disc around $\zeta_0$. Still requiring that requiring $\zeta$ is close to $\zeta_0$ we therefore get $|H_D^f|\leq 2\epsilon$ by combining the two inequalities for the Perron solution. Since $\epsilon$ was arbitrary we conclude that $\lim_{z\rightarrow\zeta_0} H^f_D(z)=0=f(\zeta_0)$.  
\end{proof}

So if the boundary is sufficiently regular we can solve the Dirichlet problem. The next criterion turns out to verify in many situations that subharmonic barriers exist.

\begin{thm}
Let $D$ be a proper subdomain of $\mathbb{C}$ and let $\partial D$ consist of a finite number of smooth boundary curves such that if $\zeta_0\in \partial D$ there exists a line segment $I\subset\mathbb{C}_\infty \backslash D$ that has one endpoint at $\zeta_0$, then there exists a subharmonic barrier a $\zeta_0$.
\end{thm}

\begin{proof}
We are going to proceed in two steps. First consider the unit disc $D=\Delta$. Then we claim that $b(z)=Re(\bar{\zeta}_0 z)-1$ is a subharmonic barrier. The observation to make is that $b(z)$ is the real part of a holomorphic function and hence is subharmonic. The two other properties of a barrier follow immediately by checking their statements.\\

For the second step we have that for $\zeta_0\in \partial D$ there exists a line segment $I\subset\mathbb{C}_\infty \backslash D$ that has one endpoint at $\zeta_0$. Since $I\subset \mathbb{C}_\infty$ is obviously connected, it follows that $\mathbb{C}_\infty \backslash I$ is simply connected. It is an immediate consequence of the Riemann mapping theorem that there exists a conformal map between two simply-connected domains of $\mathbb{C}_\infty$, which are neither $\mathbb{C}_\infty$ nor $\mathbb{C}$. So pick such a conformal map $\phi(z):\mathbb{C}_\infty \backslash I \rightarrow \Delta$ extended to the boundary such that $\phi(\zeta_0)=1$ (refer to Fact \ref{fact:cara} to see that the extension to the boundary works as claimed). Now it immediately follows that $b(z)=Re(\phi(z))-1$ is a subharmonic barrier at $\zeta_0$.
\end{proof}

It is obvious that the punctured unit disc $\Delta\backslash \{0\}$ does not satisfy the conditions of the previous theorem, which should come at no surprise as we have seen that on this domain the Dirichlet problem is not solvable.\\

Sometimes the requirement to have a line segment based at a boundary point and contained in complement is referred to as \textbf{exterior cone condition}\index{exterior cone condition}. This terminology hints on the fact that the same construction works with truncated cones in higher dimensions to get barriers. For convenience we introduce another definition describing even 'nicer' domains.

\begin{defn}
A proper subdomain $D$ of $\mathbb{C}$ will be called a \textbf{\index{Jordan domain}Jordan domain} if it is simply connected and bounded by finitely many smooth boundary curves such that $\partial D$ is a simple closed curve.  
\end{defn}

Notice that a Jordan domain satisfies the exterior cone condition, but that there are domains with smooth boundary curves satisfying the exterior cone condition and are not Jordan domains, e.g. the union of two disjoint discs.  Also we remind the reader of our convention of taking the boundary of $D$ with respect to $\mathbb{C}_\infty$, which implies that we regard e.g. the upper-half plane $H=\{z=x+iy:y>0\}$ as a Jordan domain (being bounded by a circle in $\mathbb{C}_\infty$). Summing up we have a final result on the solvability of the Dirchlet Problem.

\begin{thm} 
\label{thm:dirsol}
If $D$ is a proper subdomain of $\mathbb{C}$ such that $\partial D$ consists only of regular points (e.g. if $D$ is a domain satisfying the exterior cone condition or a Jordan domain), then the Dirichlet Problem on $D$ for a continuous, bounded function $f:\partial D\rightarrow\mathbb{R}$ is uniquely solvable and the solution is given by the Perron solution $H^f_D(z)$.  
\end{thm}

We are going to proceed to take up the Dirichlet problem from a different viewpoint in the next section. 

\subsection{Harmonic Measure and Brownian Motion} 

The next application we consider is drawn from probability theory and illustrates just one of the large number of interconnections between the material developed so far and probability.\\

We are mainly going to follow for the part on Brownian motion \cite{11} and \cite{12}. Notation and standard terminology for probabilistic terms follows \cite{12} and \cite{13}; if there exist differences between analysts' and probabilists' terms we remark those in brackets if we use them for the first time. First we recall the definition of Brownian motion.\\ 

\begin{defn} 
A sequence of random variables (=measurable functions) $(B_t)_{t\geq 0}$ on a probability measure space $(\Sigma,M,\mu)$ with values in $\mathbb{R}^n$ is called a standard \textbf{\index{Brownian motion}Brownian motion} if $t\mapsto B_t$ is continuous and 
\begin{description}
	\item[(a)] $B_0=0$ almost surely (=almost everywhere)
	\item[(b)] for every $0\leq t_1\leq \ldots \leq t_k$ the increments $(B_{t_1}-B_{t_0},\ldots,B_{t_k}-B_{t_{k-1}})$ are independent
	\item[(c)] for all $t,s\geq 0$ and $t<s$, $B_{t+s}-B_{t}$ is a Gaussian random variable with mean $0$ and covariance matrix $s\cdot Id$.
\end{description}
\end{defn}

We are interested in the case $n=2$, so we use Brownian motion taking values in $\mathbb{R}^2$. If $D$ is a proper subdomain of $\mathbb{C}$ such that $0\in D$ we define a stopping time $T=\inf\{t\geq 0:B_t\not\in D\}$. We ask the question\\

\textbf{What is the distribution of a Brownian motion $B_t$ leaving a domain $D$ at time $T$?}\\

More precisely, we want to find the distribution of the random variable $B_T$. For this purpose we first establish a connection between the Dirchlet problem and Brownian motion. Let $\mathbb{E}[X|Y]$ denote the conditional expectation of a random variable $X$ with respect to $Y$.

\begin{thm}
\label{thm:probdi}
If $D\subset \mathbb{C}$ is a domain such that $\partial D$ consists only of regular points and $f:\partial D\rightarrow \mathbb{R}$ is continuous and bounded then
\begin{eqnarray}
u(z)=\mathbb{E}\left[f(B_T)| B_0=z\right]
\end{eqnarray}
solves the Dirichlet problem.
\end{thm}
 
This might be quite surprising if one sees the result for the first time, therefore we are going to proceed to outline a sketch of the proof. For notational convenience we set $\mathbb{E}\left[f(B_T)| B_0=z\right]=\mathbb{E}_z\left[f(B_T)\right]$. We need two elementary results from probability theory.

\begin{fact}[\index{Optional Stopping Theorem}Optional Stopping Theorem, \cite{13}] 
\label{fact:ost}
Let $(X_t)_{t\geq 0}$ be a martingale with values in $\mathbb{R}$, which is  uniformly integrable (recall: for any $\epsilon>0$ $\exists \delta>0$ such that $\mathbb |\mathbb{E}[ X_t 1_S]|<\epsilon$ whenever $\mu(S)<\delta$ for any $t\geq 0$). Then it follows that $\mathbb{E}[X_t|F_s]=X_s$ where $F_s$ is the $\sigma$-algebra generated by $X_s$.   
\end{fact}

\begin{fact} 
\label{fact:heatmg}
Let $(B_t)_{t\geq 0}$ be a Brownian motion and $g:\mathbb{R}^2\rightarrow \mathbb{R}$ in $C^2(\mathbb{R}^2)$ with bounded derivatives then
\begin{eqnarray}
M_t=g(B_t)-g(B_0)-\int_0^t\Delta g(B_s)ds \qquad t\geq 0
\end{eqnarray}
is a martingale.
\end{fact}

Note that we do not have to show the existence of a solution to the Dirichlet problem again, hence we concentrate on the following partial result showing that if a solution exists then we can write it terms of Brownian motion.

\begin{thm} Let $D$ be a proper subdomain of $\mathbb{C}$ with boundary consisting only of regular points and let $f:\partial D\rightarrow \mathbb{R}$ be continuous and bounded. If ${T<\infty}$ almost surely,
\begin{eqnarray}
u(z)=\mathbb{E}_z\left[f(B_T)1_{T<\infty}\right]
\end{eqnarray}
and $v$ is any bounded solution to the Dirichlet problem then $u=v$ 
\end{thm}

\begin{proof}(Sketch) Define $D_r=\{z\in D: |z|<r \text{ and } d(z,\partial D)>1/r\}$ for $r\geq 1$. Let $T_n$ be the first exit time for $D_n$. Now modify $v$ to obtain $\tilde{v}_n$ such that $\tilde{v}_n\in C^2$ with bounded derivatives and $\tilde{v}_n=v$ on $D_n$ and use Fact \ref{fact:heatmg} to conclude that  
\begin{eqnarray}
M_t=\tilde{v}_n(B_t)-\tilde{v}_n(B_0)-\int_0^t \Delta \tilde{v}_n(B_s) ds\notag
\end{eqnarray}  
is a martingale. Consider stopping the martingale $M_t$ at $T_n$, then $M_{\min(t,T_n)}=v(B_{\min(t,T_n)})-v(B_0)$ since $v=\tilde{v}_n$ on $D_n$ and $\Delta v=0$ on $D$ as $v$ solves the Dirichlet problem. Observe that $M_{\min(t,T_n)}$ is bounded since $v$ is bounded and therefore it is uniformly integrable, so we use optional stopping (Fact \ref{fact:ost}) at $T_n$ to obtain for all $z\in D$
\begin{eqnarray}
\mathbb{E}_z \left[M_{T_n}|F_0\right]=\mathbb{E}_z\left[M_{T_n}\right]=\mathbb{E}_z\left[M_0\right]=0\notag
\end{eqnarray} 
Hence it follows that $0=\mathbb{E}_z\left[v(B_{T_n})\right]-v(x)$. By continuity of Brownian motion $B_{T_n}\rightarrow B_T$ as $n\rightarrow\infty$ and since $v$ is bounded, we can use the dominated convergence theorem to conclude
\begin{eqnarray}
v(z)=\lim_{n\rightarrow\infty} \mathbb{E}_z\left[v(B_{T_n})\right]=\mathbb{E}_z\left[v(B_T)\right]=\mathbb{E}_z\left[f(B_{T_n})\right]\notag
\end{eqnarray}
where the last equality follows as we have stopped the process at $T$ on the boundary and $v$ is a solution so $v=f$ on $\partial D$. 
\end{proof}

Notice the crucial role that the martingales $M_{\min(t,T_n)}$ played in the proof for 'approximating' the solution $v$ similar to subharmonic functions for the Perron solution. Although this is a rather vague similarity, the technique is no coincidence and it turns out that one can develop basically all of potential theory from a probabilistic viewpoint using Brownian motion and martingales - for details see \cite{10}.\\

A refined analysis (see also \cite{10}) shows that the hypothesis $\{T<\infty\}$ can be removed to yield Theorem \ref{thm:probdi}. Notice that we have then $H^f_D(z)=\mathbb{E}_z\left[f(B_T)\right]$, where $H^f_D$ is the Perron solution, suppose further that we can find the probability density function $p$ of the random variable $B_T$ then we have

\begin{eqnarray}
H^f_D(z)=\mathbb{E}_z\left[f(B_T)\right]=\int_{\partial D}f(z)p(z)dm(z)
\end{eqnarray}

This shows that we can obtain a 'nice' formula for the solution in terms of a measure $p(z)dm(z)$ and the question is if the same concept exists also in the language of 'classical' potential theory.

\begin{defn} Let $D$ be proper subdomain of $\mathbb{C}$ such that $\partial D$ consists only of regular points and let $\cal{B}$$(\partial D)$ denote the Borel $\sigma$-algebra on $\partial D$. A \textbf{\index{harmonic measure}harmonic measure} for $D$ is a function $\omega_D:D \times \cal{D}$$(\partial D)\rightarrow [0,1]$ such that for all $z\in D$, the map $B\mapsto \omega_D(z,B)$ is a probability measure on $\cal{B}$$(\partial D)$ and if $f:\partial D\rightarrow \mathbb{R}$ is continuous and bounded then $H^f_D=P^f_D$ where $P^f_D$ is the \textbf{\index{generalized Poisson integral}generalized Poisson integral} given by
\begin{eqnarray}
P^f_D(z)=\int_{\partial D} f(s) d\omega_D(z,s) \qquad (z\in D)
\end{eqnarray}
\end{defn}

Harmonic measure is exactly the concept one would expect after looking at the probabilistic solution of the Dirichlet problem. First we have to show that such an object actually exists. For the development of the main ideas about harmonic measure we follow \cite{14} and modify the proof of an existence result from \cite{4}. 

\begin{thm} Let $D$ be a proper subdomain of $\mathbb{C}$ such that $\partial D$ is regular, then there exists a unique harmonic measure $\omega_D$ for $D$. 
\end{thm} 

\begin{proof} 
The first step is to show that $f\mapsto H^f_D$ defines a positive linear functional. Note that we have $H^f_D=\mathbb{E}_z[f(B_T)]$ for all continuous bounded functions $f$. So consider two continuous bounded functions $f_1,f_2$ on $\partial D$. By linearity of expectations (=linearity of integrals) we have
\begin{eqnarray}
\mathbb{E}_z[f_1(B_T)]+\mathbb{E}_z[f_2(B_T)]=\mathbb{E}_z[f_1(B_T)+f_2(B_T)]=\mathbb{E}_z[(f_1+f_2)(B_T)]\notag
\end{eqnarray}
Also notice that if $f=1$, then $\mathbb{E}_z[1]=1$ and if $f\geq 0$ then $\mathbb{E}_z[f(B_T)]\geq 0$. Hence we have the positive linear functional $f\mapsto H^f_D$. By the Riesz representation theorem, there exists a unique Borel probability measure $\mu_z$ such that 
\begin{eqnarray}
H^f_D(z)=\int_{\partial D} f(z) d\mu_z\notag
\end{eqnarray}   
Now we simply define for $z\in D$ and $B\in \cal{B}$$(\partial D)$ the harmonic measure by $\omega_D(z,B)=\mu_z(B)$ to obtain the result.
\end{proof}

Notice that with a little bit more work we could have shown directly that $H^f_D$ is a positive linear functional, but it seems reasonable to gain some confidence in working with the probabilistic formula here.

\begin{ex} In one case we have already calculated the harmonic measure, namely if $D=\Delta(0,R)$ we have the Poisson integral formula (see Theorem \ref{thm:pifdisc}) for $w=Re^{is}$
\begin{eqnarray}
H^f_{\Delta(0,R)}(z)=P^f_{\Delta(0,R)}(z)= \frac{1}{2\pi} \int_0^{2\pi} \frac{|w|^2-|z|^2}{|w-z|^2}f(Re^{is})ds\notag
\end{eqnarray}
Now one can immediately spot the harmonic measure for $B\in \cal{B}$$(\partial \Delta(0,R))$, namely
\begin{eqnarray}
\omega_{\Delta(0,R)}(z,B)=\int_B P(z,Re^{is})\frac{ds}{2\pi}
\end{eqnarray} 
where $P(z,w)$ is the Poisson kernel, which we have thereby (re-)interpreted - disregarding the normalizing factor $2\pi$ - as a Radon-Nikodym derivative of harmonic measure.
\end{ex}

Our next goal is calculate harmonic measure for more general domains than just a disc. For this we need a remarkable result from complex analysis.

\begin{fact}[\index{Caratheodory's Theorem}Caratheodory's Theorem, \cite{14}]
\label{fact:cara} 
Let $D$ be a Jordan domain and let $\phi$ be a conformal map from $\Delta(0,1)=\Delta$ to $D$, then $\phi$ has a continuous extension to the boundary $\overline{\Delta}$. Furthermore this extension is one-to-one from $\overline{\Delta}$ to $\overline{D}$.
\end{fact}

Note that the conformal map obviously exists by the Riemann mapping theorem. A continuous extension to the boundary will obviously turn out to be useful in the context of the Dirchlet problem.

\begin{thm}
\label{thm:confinv}
If $D$ is a Jordan and $f:\partial D\rightarrow \mathbb{R}$ is continuous and bounded, then for a conformal map $\phi:D\rightarrow \Delta$ (extended by Caratheodory's theorem) we have that for $w=\phi(z)$
\begin{eqnarray}
g(z)=\frac{1}{2\pi}\int_0^{2\pi} f\circ \phi^{-1} (e^{is})\frac{1-|w|^2}{|e^{is}-w|^2}ds
\end{eqnarray}
solves the Dirichlet problem.
\end{thm}

\begin{proof}
To show that $g(\phi^{-1}(w))$ is harmonic simply calculate the derivatives by the chain rule, use the fact that $g$ and $Re(\phi^{-1})$ are harmonic and use the Cauchy-Riemann equations for the remaining terms as $\phi^{-1}$ is holomorphic. Also $g(z)$ extends continuously to the boundary by combining Caratheodory's theorem with the Poisson integral formula for the disc.  
\end{proof}
  
Regarding harmonic measure we thus have obtained a rather remarkable result about conformal invariance.

\begin{thm}[\index{conformal invariance}Conformal Invariance]
If $D$ is a Jordan domain, $\phi$ is a conformal map from $D$ onto the unit disc $\Delta$ (extended by Caratheodory's theorem) and $B$ is a Borel measurable subset of $\partial D$, then 
\begin{eqnarray}
\omega_D(z,B)=\omega_\Delta (\phi(z),\phi(B))=\int_{\phi(B)}\frac{1-|\phi(z)|^2}{|e^{is}-\phi(z)|^2}\frac{ds}{2\pi}
\end{eqnarray}
\end{thm}

\begin{proof} 
Immediate from Theorem \ref{thm:confinv}.
\end{proof}

Finally we return to the initial question of how to find the probability density function of a Brownian motion $B_T$ stopped a boundary of a domain $D$. Clearly for the disc we know the answer from the Poisson integral formula, now we are focusing on the upper half-plane. 

\begin{thm} 
Let $H=\{z=x+iy \in \mathbb{C}: y>0\}$ and let $B$ be a Borel subset of $\mathbb{R}$, then 
\begin{eqnarray}
\omega_H(x+iy,B)=\frac{1}{\pi}\int_B \frac{y}{(x-t)^2+y^2}dt \qquad x+iy=z\in H
\end{eqnarray}
\end{thm}

\begin{proof}
Let $\phi:H\rightarrow\Delta$ be the conformal map from the upper half-plane to the unit disc given by
\begin{eqnarray}
\phi(z)=\frac{z-i}{z+i}\notag
\end{eqnarray}
By conformal invariance of harmonic measure we have a straightforward calculation 
\begin{eqnarray}
\omega_H(z,B)&=& \omega_\Delta(\phi(z),\phi(B))=\int_{B}\frac{1-|\phi(z)|^2}{|\phi(s)-\phi(z)|^2}|\phi'(s)|\frac{ds}{2\pi}\notag\\
&=&\int_{B}\frac{Im(z)}{|z-s|^2}\frac{ds}{\pi}=\int_{B}\frac{y}{(x-s)^2+y^2}\frac{ds}{\pi}\notag
\end{eqnarray}
\end{proof}

Getting a final result for upper half-plane is simply obtained by putting the previous results together.

\begin{thm} 
The probability distribution of $B_T$ for the upper half-plane is given by a \index{Cauchy distribution}Cauchy distribution.
\end{thm}

\begin{proof}
We can apply the fact that we have a probabilistic solution formula $\mathbb{E}_z[f(B_T)]$, which must be equal to the generalized Poisson integral obtained via the harmonic measure to derive the probability density function $p$ of $B_T$, namely for $f$ bounded and continuous on $\mathbb{R}$
\begin{eqnarray}
P_D^f(x,y)=\int_{\mathbb{R}}f(s)\frac{1}{\pi}\frac{y}{(x-s)^2+y^2}ds=\mathbb{E}_{(x,y)}[f(B_T)]=\int_\mathbb{R}f(s)p(x,y,s)ds
\end{eqnarray} 
where we can read off p(x,y,s) as the probability density function of a Cauchy distribution (sometimes called a Cauchy distribution with 'location parameter' $s$ and 'scale parameter' $y$).
\end{proof}

Notice that the only data we need to find the harmonic measure for a simply connected subdomain are the conformal map and the Poisson formula for the unit disc. It is possible to use characteristic functions of distributions (= Fourier transforms of measures) to obtain the Cauchy distribution for the upper half-plane (see \cite{9}), but this approach is clearly not as flexible as conformal maps.

\subsection{Green's Function and Growth of Polynomials}

Probability theory is not the only field, where the ideas we developed so far turn out to be useful. To demonstrate this we now shift our attention to polynomials. Consider the unit disc $\Delta=\Delta(0,1)$ and let $p_n(z):\mathbb{C}\rightarrow\mathbb{C}$ be a polynomial of degree $n\in\mathbb{N}$. We ask the question\\

\textbf{What is the growth rate of $p_n$ on $\mathbb{C}$ given its values on $\Delta$?}\\

Again subharmonic functions and many other concepts developed in the last sections will turn out to be very useful answering this question, but first we need another concept. In the following we have used \cite{15} and \cite{8} as our main sources.

\begin{defn}
Let $D$ be a bounded Jordan domain. Fix $\zeta\in D$ and let $f:\partial D\rightarrow \mathbb{R}$ be given by $f(z)=\log |z-\zeta|$. Then define \textbf{\index{Green's function}Green's function} $g_D(z,\zeta)$ for $D$ with a pole at $\zeta$ by 
\begin{eqnarray}
g_D(z,\zeta)=H^f_D(z)-\log |z-\zeta|
\end{eqnarray}  
where $H^f_D(z)$ denotes the Perron solution to the Dirichlet problem on $D$ as previously.
\end{defn}

Notice that the way we decided to introduce Green's function is unusual, but has the advantage of a concrete formula, showing directly that Green's function can viewed as a harmonic function on $D\backslash \{\zeta\}$, which is zero on the boundary. Also our definition turns out be consistent with the axiomatic one.

\begin{thm}[\index{Axioms for Green's function}Axioms for Green's function]
If $D$ is a bounded Jordan domain then a Green's function with pole at $\zeta$ is uniquely characterized by the following three properties
\begin{description}
	\item[(a)] $g_D(z,\zeta)$ is non-negative on $D$ and harmonic on $D\backslash \{\zeta\}$
	\item[(b)] $g_D(z,\zeta)+\log|z-\zeta|$ is harmonic at $z=\zeta$
	\item[(c)] $g_D(z,\zeta)\rightarrow 0$ as $z\rightarrow \partial D$
\end{description}
\end{thm}
   
\begin{proof} 
Fix any $\zeta$ in $D$. Suppose $g_D(z,\zeta)$ and $\tilde{g}_D(z,\zeta)$ are two functions satisfying (a)-(c). Then $h(z,\zeta)=\tilde{g}_D(z,\zeta)-g_D(z,\zeta)$ is harmonic. Hence the maximum principle for harmonic functions forces $h$ to be constant on $D$, but $h(z,\zeta)\rightarrow 0$ for $z\rightarrow \partial D$ so we have $h\equiv 0$, which means $\tilde{g}_D(z,\zeta)=g_D(z,\zeta)$ since $\zeta$ was arbitrary.   
\end{proof}

Notice that our definition of Green's function above trivially satisfies all three conditions (a)-(c) (if in doubt, the reader is advised to consult \ref{thm:logsub}, \ref{defn:perfam}, \ref{defn:persol}, \ref{thm:reh} and \ref{thm:dirsol}). It is interesting to connect Green's function to harmonic measure. For this purpose we need another version of Green's formula.

\begin{fact}[\index{Green's formula - another version}Green's formula - another version, \cite{8}]
Let $D$ be a bounded Jordan domain and let $u$ be a smooth real-valued function on $\overline{D}$. Furthermore fix $\zeta\in D$ and consider $v:\overline{D}\rightarrow \mathbb{R}$ such that $v$ is harmonic on $D\backslash \{\zeta\}$, extends smoothly to be zero on $\partial D$ and $v(z)+\log|z-\zeta|$ is harmonic at $\zeta$ then 
\begin{eqnarray}
u(\zeta)=-\frac{1}{2\pi}\int_D v \Delta u \text{ }dm -\frac{1}{2\pi}\int_{\partial D}u\frac{\partial v}{\partial n}ds
\end{eqnarray} 
where $\frac{\partial}{\partial n}$ denotes the derivative in the direction of the inner normal vector. 
\end{fact}

Note that this version of Green's formula can be proven by removing a small disc around the pole $\zeta$, using the version of Green's formula we gave earlier (see Fact \ref{fact:gr1}) and a few direct estimates.
 
\begin{thm} 
Let $D$ be a bounded Jordan domain. Fix $\zeta\in D$ and let $g_D(z,\zeta)$ be the Green's function for $D$ (with pole at $\zeta$), then for any Borel measurable set $B\in \cal{B}$$(\partial D)$ we have
\begin{eqnarray}
-\frac{1}{2\pi}\int_B \frac{\partial g_D}{\partial n}(z,\zeta)ds=\omega_D(\zeta,B)
\end{eqnarray}
where $ds$ denotes the integral of $z$ along $B$.
\end{thm}

\begin{proof}
Obviously the idea is to apply Green's formula, so let $u(\zeta)$ be any harmonic function on $\overline{D}$. Then we use Green's formula for $v=g_D(z,\zeta)$ to get
\begin{eqnarray}
u(\zeta)=-\frac{1}{2\pi}\int_D g_D \Delta u \text{ }dm -\frac{1}{2\pi}\int_{\partial D}u\frac{\partial g_D}{\partial n}ds= -\frac{1}{2\pi}\int_{\partial D}u\frac{\partial g_D}{\partial n}ds\notag
\end{eqnarray}
since $u$ is a solution to the Dirichlet problem and by uniqueness of the solution we have
\begin{eqnarray}
u(\zeta)=H^{u|_{\partial D}}_D(\zeta)=\int_{\partial D}u(z)d\omega(\zeta,z)\notag
\end{eqnarray}
which immediately yields the result.
\end{proof}

It is straightforward to extend the concept of Green's function if we are working with unbounded simply-connected domains as the following result about conformal mappings shows.

\begin{thm} 
Let $D$ be a simply-connected proper subdomain of $\mathbb{C}$ and fix $\zeta \in D$ then there exists a Green's function $g_D(z,\zeta)$ on $D$, i.e. $g_D(z,\zeta)$ satisfies the axioms for Green's function. 
\begin{description}
	\item[(a)] $g_D(z,\zeta)$ is non-negative on $D$ and harmonic on $D\backslash \{\zeta\}$
	\item[(b)] $g_D(z,\zeta)+\log|z-\zeta|$ is harmonic at $z=\zeta$
	\item[(c)] $g_D(z,\zeta)\rightarrow 0$ as $z\rightarrow \partial D$
\end{description}
\end{thm}

\begin{proof} 
Let $\phi:D\rightarrow \Delta$ be the conformal map obtained from the Riemann mapping theorem onto the unit disc. Note that we can construct $\phi$ such that $\phi(\zeta)=0$ and $\phi'(\zeta)>0$ as we can post-compose conformal mappings with self-maps of the disc. Then define
\begin{eqnarray}
g_D(z,\zeta)=-\log|\phi(z)|\notag
\end{eqnarray}
Now we have to check properties (a)-(c) to confirm that this construction works. Since $|\phi(z)|<1$ we obtain $g_D\geq 0$. Also $\lim_{z\rightarrow \partial D} g_D(z,\zeta)=0$ since $|\phi(z)|\rightarrow 1$ as $z\rightarrow\partial D$. By the construction of $\phi$ we have that $z\mapsto \frac{\phi(z)}{z-\zeta}$ is holomorphic and non-zero at $z=\zeta$ by considering limits. So we can take the logarithm to get 
\begin{eqnarray}
\log \left|\frac{\phi(z)}{z-\zeta}\right|=-(g_D(z,\zeta)+\log|z-\zeta|)\notag
\end{eqnarray}
Hence $g_D(z,\zeta)+\log|z-\zeta|$ is harmonic at $\zeta$ and since $g_D(z,\zeta)$ is obviously harmonic on $D\backslash \{\zeta\}$ we have the result.
\end{proof}

Notice that so far we have only dealt with Green's functions for bounded Jordan domains and simply-connected proper subdomains of $\mathbb{C}$. It turns out that the notion of Green's function can be extended even further, but we cannot expect that we always obtain a bounded Green's function.

\begin{ex}
To give a motivation for the last statement, consider $D=\mathbb{C}$ and $\zeta=0$. Define $D_n=\{z\in\mathbb{C}:|z|<n\}$. Then on $D_n$ we obviously have as a Green's function
\begin{eqnarray}
g_{D_n}(z,0)=\log n-\log|z|\notag
\end{eqnarray}
But if we try to deduce from this a Green's function for $\mathbb{C}$ we fail in the sense that $g_{D_n}\rightarrow\infty$ for $n\rightarrow \infty$.
\end{ex}

But using the ideas from the conformal mapping above we can extend the notion of a Green's function to include the case, when we have a pole at $\zeta=\infty$ and $D$ is the complement of the closed unit disc, i.e. $D=\mathbb{C}\backslash \overline{\Delta}$. We set
\begin{eqnarray}
g_D(z,\infty)=g_\Delta(\frac{1}{z},0)
\end{eqnarray}
It is easy to check that this function satisfies all the axioms for Green's function. Now we are ready to get back to the growth of polynomials.

\begin{thm}
Let $p_n(z)$ be a polynomial of degree $n\in\mathbb{N}$ on $\mathbb{C}$ and let $D= \mathbb{C} \backslash \overline{\Delta}$. If
\begin{eqnarray}
\|p_n\|_{\overline{\Delta}} := \max_{z\in\overline{\Delta}} |p_n(z)| \leq M
\end{eqnarray}
then we obtain a growth condition on $\mathbb{C}$ given by
\begin{eqnarray}
|p_n(z)|\leq M e^{ng_D(z,\infty)} \qquad \forall z\in\mathbb{C}
\end{eqnarray}
\end{thm}

\begin{proof}
Define $v(z)=\log|p_n(z)|-ng_D(z,\infty)$ and note that $v$ is subharmonic on $D$ since the logarithm of a holomorphic function is subharmonic. Also $v$ is harmonic in a neighborhood of $\infty$ since $p_n$ has finitely many zeros and $g_D(z,\infty)$ is harmonic on $D$. Now we need to establish an estimate for $v$ and for this observe that
\begin{eqnarray}
\lim_{w\rightarrow\partial D}v(w)\leq \log M\notag
\end{eqnarray}
and therefore the maximum principle for subharmonic functions implies that $v(z)\leq \log M$ on D. Upon taking exponentials of $\log|p_n|-ng_D(z,\infty)\leq \log M$ we can use the maximum principle for holomorphic functions (!) to conclude that $|p_n(z)|\leq M e^{ng_D(z,\infty)}$ . 
\end{proof}

It remains to remark that the result can be extended if one does use a less explicit construction of Green's function, but this approach does not lead to any gain in pointing out that the concept of a subharmonic function has been essential in the preceeding proof.

\begin{fact}[\index{Bernstein-Walsh}Bernstein-Walsh, \cite{15}]
Let $p_n(z)$ be a polynomial of degree $n\in\mathbb{N}$ on $\mathbb{C}$ and let $K$ be a compact subset of $\mathbb{C}$. Then we obtain the following implication
\begin{eqnarray}
\|p_n\|_K := \max_{z\in K} |p_n(z)| \leq M \quad \Longrightarrow \quad |p_n(z)|\leq M e^{ng_{\mathbb{C}\backslash K}(z,\infty)} \quad \forall z\in\mathbb{C} 
\end{eqnarray}
\end{fact} 

Clearly this supposes that we can construct a Green's function for the complement of a compact set with a pole at $\infty$, but this turns out to be possible without problems under slightly weakened assumptions of the Green's function axioms (see \cite{5}).

\subsection{Equilibrium Measures and Physics}

Subharmonic functions and potential theory grew largely out of considerations from physics and although the influence of the initial ideas is more of historic value today, it provides interesting insight, why certain concepts we have met so far, have been developed. We are going to give a brief - mostly heuristic - introduction to the initial ideas from electrostatics underlying potential theory mainly based on \cite{19} and partly on \cite{20}.\\

Let $c$ be a (positive) point charge in the plane at $w=(x_0,y_0)$. Suppose that we found by experiment that the electric field at $z=(x,y)$ has direction from $w$ to $z$ and magnitude $c/|z-w|$. In classical vector notation this means that the electric field $\vec{F}$ is given by  
\begin{eqnarray}
\vec{F}(z)=\frac{c}{|z-w|}\frac{z-w}{|z-w|}
\end{eqnarray}  
then we observe that for a so-called potential function $\varphi(z)$ given by 
\begin{eqnarray}
\varphi(z)=-c\log|z-w|
\end{eqnarray}
we can express the force field as a negative gradient of the potential (i.e. $\vec{F}$ is 'conservative')
\begin{eqnarray}
-\nabla \varphi(z)=c\left(\frac{\partial \log|z-w|}{\partial x},\frac{\partial \log|z-w|}{\partial y}\right)=\vec{F}(z)\notag
\end{eqnarray}
Notice carefully that the force field and the potential function possess 'singularities' at $z=w$. Suppose now that we do not only have one, but several charges $\{c_i\}_{i=1}^n$ placed at $w_i$, then experiments show that the electric fields indeed simply add up to give
\begin{eqnarray}
\vec{F}(z)=\sum_{i=1}^n\frac{c_i}{|z-w_i|^2}(z-w_i) \qquad \text{and} \qquad \varphi(z)=-\sum_{i=1}^n c_i\log|z-w_i|\notag
\end{eqnarray} 
Taking this idea one step further we can suppose that charge is continuously distributed in some region $K$ given by a function $\rho:\mathbb{R}^2\rightarrow\mathbb{R}$ such that the total charge in a subset $B$ of $K$ is now given by
\begin{eqnarray}
\int_B p(z)dm(z)\notag
\end{eqnarray}
Now if we divide up $K$ into $n$ sufficiently small pieces $B_i$, the charge in $B_i$ can be approximated by $\rho(w_i)m(B_i)$, where $w_i$ is a point in $B_i$. Hence we get an approximation $\vec{F}^n$ converging to an integral
\begin{eqnarray}
\vec{F}^n(z)=\sum_{i=1}^n\frac{m(B_i)\rho(w_i)}{|z-w_i|^2}(z-w_i)\stackrel{n\rightarrow \infty}{\rightarrow} \int_K \frac{\rho(w)}{|z-w|^2}(z-w)dm(w) 
\end{eqnarray} 
The potential function is obviously given by
\begin{eqnarray}
\varphi(z)=-\int_K \rho(w)\log|z-w|dm(w) 
\end{eqnarray}
We can interpret $\rho(w)dm(w):=\mu$ as a Borel measure to obtain the potential $p_\mu$ as defined in \ref{defn:logpot} except for a negative sign, which we already argued is just introduced for mathematical convenience to make the potential subharmonic. We remark that the same considerations can be also employed for the three-dimensional case by noticing that in this case the force is given $c/|z-w|^2$ to derive the Newtonian potential (see Definition \ref{defn:newton}).\\

We can proceed and integrate the potential to get
\begin{eqnarray}
\int_K \varphi(z) \rho(z)dm(z)=-\int_K \int_K \rho(w)\log|z-w|dm(w) \rho(z)dm(z)=-\int p_\mu(z)d\mu(z)
\end{eqnarray} 
which is the energy contained in $K$ due to the charge distribution $\rho$ (or respectively the measure $\mu$). Since we could suspect from considerations based on principles of physics that the charge distribution changes as long as there is a force present and the principle of minimized energy for the region $K$ is not satisfied, we ask the question\\

\textbf{Is there an equilibrium charge distribution (measure) for the region $K$?}\\

Ending the aside from electrostatics we try to answer this question from a mathematical viewpoint following \cite{4} and \cite{15}.

\begin{defn}
If $\mu$ is a finite and compactly supported Borel measure, then we define its \textbf{\index{energy}energy} $I(\mu)$ by
\begin{eqnarray}
I(\mu)=-\int_\mathbb{C}\int_\mathbb{C} \log|z-w|d\mu(w)d\mu(z)=-\int_\mathbb{C}p_\mu(z)d\mu(z)
\end{eqnarray} 
\end{defn} 

The definition is fairly obvious from previous considerations and we look at a special case first to understand the role of the support of the measure $\mu$ a little bit better.

\begin{ex} 
Suppose that $supp(\mu)=\{z_0\}$ and $\mu(z_0)=1$ so that we deal with a measure of a single atom then we have
\begin{eqnarray}
I(\mu)=-\int_\mathbb{C}\int_\mathbb{C} \log|z-w|d\mu(w)d\mu(z)=-\int_\mathbb{C} \log|z-z_0|d\mu(w)d\mu(z)=\infty\notag
\end{eqnarray}  
so a point does not support any finite measure of finite energy, which is reminiscent of the situation of a point charge.
\end{ex}

Borel sets, which only admit measures of infinite energy are extremely important since they can be viewed as potential theoretic equivalent to sets of measure zero.

\begin{defn} 
Let $B$ be a Borel subset of $\mathbb{C}$, then $B$ is called \textbf{\index{polar}polar} if $I(\mu)=\infty$ for all finite and compactly supported, non-zero measures $\mu$ with $supp(\mu)\subseteq B$. Furthermore we say that a property holds quasi-everywhere for a Borel set $S\subseteq \mathbb{C}$ if it holds for $S\backslash B$, where $B$ is a Borel polar set.  
\end{defn}

The natural question to investigate first is the relation between polar sets and sets of Lebesgue measure zero. To do this we need a preliminary result.

\begin{lem} 
Let $\mu$ be a finite Borel measure on $\mathbb{C}$ with compact support and suppose $I(\mu)<\infty$, then $\mu(E)=0$ for any Borel polar subset $E$ of $\mathbb{C}$. 
\end{lem}   

\begin{proof} 
We argue by contraposition, i.e. let $B$ be a Borel set and $\mu(B)>0$ and we want to show that $B$ cannot be polar, i.e. there exists a measure compactly supported in $B$ with finite energy. To construct this measure notice first that any finite Borel measure is automatically regular, so there exists a compact set $K\subseteq B$ such that $\mu(K)>0$. Restricting $\mu$ to $K$, we have a measure $\mu|_K$ with $supp(\mu|_K)\subseteq B$. Now set $d=diam(supp(\mu))$ and notice that $|z-w|/d<1$ if $z,w\in supp(\mu)$, so that
\begin{eqnarray}
I(\mu|_K)&=&-\int_K \int_K \log|z-w|d\mu|_K(z)d\mu|_K(w)=-\int_K \int_K \log\frac{|z-w|}{d}d\mu(z)d\mu(w)-\mu(K)^2 \log d\notag\\
&\leq&-\int_\mathbb{C} \int_\mathbb{C} \log\frac{|z-w|}{d} d\mu(z)d\mu(w)-\mu(K)^2 \log d=I(\mu)+\mu(\mathbb{C})^2\log d -\mu(K)^2\log d <\infty\notag
\end{eqnarray} 
and we can indeed conclude that $B$ is non-polar.
\end{proof}

The relation between polar sets and Lebesgue measure can now be answered for the case, when we have a Borel polar set.

\begin{thm}
If $E$ is a Borel polar set in $\mathbb{C}$ then $m(E)=0$ (i.e. $E$ has Lebesgue measure 0). 
\end{thm} 

\begin{proof}
The idea is to restrict the Lebesgue measure to a disc and show that $I(m|_{\Delta(0,r)})<\infty$ for all $r>0$. Then we can apply the previous lemma to deduce that $m|_{\Delta(0,r)}(E\cap \Delta(0,r))=0$ for any Borel polar set $E$ and on letting $r\rightarrow\infty$ we get the result. To estimate the energy we start with the potential and apply a similar procedure as for the Lemma, so let $\mu=m|_{\Delta(0,r)}$, then for $z\in \Delta(0,r)$ 
\begin{eqnarray}
-p_\mu(z)&=&-\int_{\Delta(0,r)} \log\left|\frac{z-w}{2r}\right|dm|_{\Delta(0,r)}-\pi r^2\log(2r)\notag\\
&\leq& -\int_0^{2\pi}\int_0^{2r} \log\left(\frac{s}{2r}\right)sdsdt -\pi r^2\log(2r)\notag\\
&=& 2\pi r^2-\pi r^2\log(2r)\notag
\end{eqnarray} 
Now this result can be used to estimate the energy 
\begin{eqnarray}
I(\mu)=-\int_{\Delta(0,r)}p_\mu(z)dm(z)\leq \left(2\pi r^2+\pi r^2 \log (2r) \right) \pi r^2< \infty\notag
\end{eqnarray}
\end{proof}

This result is clearly in accordance with the fact that a point is polar. We can continue to formalize the concept of an equilibrium distribution. From the viewpoint of physics this equilibrium distribution should minimize the total energy and this turns out to be mathematically reasonable as well.

\begin{defn}
Let $K$ be a compact subset of $\mathbb{C}$ and let $\cal{P}$$(K)$ denote the Borel probability measures on $K$. If there exists $\nu\in \cal{P}$$(K)$ such that
\begin{eqnarray}
I(\nu)=\inf_{\mu\in\text{$\cal{P}$$(K)$}} I(\mu)
\end{eqnarray}
then $\nu$ is called an \textbf{equilibrium measure}\index{equilibrium measure} for $K$.
\end{defn}

For polar sets the definition is basically void as any measure has infinite energy, but for non-polar sets the search for an equilibrium measure is non-trivial. First we need a preliminary result, which has an interesting proof by means of elementary functional analysis.

\begin{lem}
If the sequence of measures $\mu_n$ converges in the $weak^{*}$-topology of $\cal{P}$$(K)$ to a measure $\mu$, then $\liminf_{n\rightarrow \infty} I(\mu_n) \geq I(\mu)$.
\end{lem}

\begin{proof}
Given two continuous functions $f,g$ on $K$ we obtain from the definition of $\text{weak}^*$-convergence that 
\begin{eqnarray}
\int_K \int_K f(z)g(w)d\mu_n(z)d\mu_n(w)\stackrel{n\rightarrow \infty}{\rightarrow}\int_K \int_K f(z)g(w)d\mu(z)d\mu(w) \notag
\end{eqnarray}
But the Stone-Weierstra$\text{\ss}$ theorem implies that any continuous function of two complex variables, say $\phi(z,w)$ can be approximated by sums of products of single-variable functions of the form $\bar{\phi}(z,w)=\sum_{i} f_i(z) g_i(w)$, where $f$ and $g$ are continuous.
Indeed, the family of functions $\{\bar{\phi}\}$ is immediately seen to be an algebra with identity and seperation of points is provided by single-(complex)-variable polynomials, so that the Stone-Weierstra$\text{\ss}$ theorem applies. Hence we can apply the initial observation about $\text{weak}^*$-convergence to obtain
\begin{eqnarray}
\int_K \int_K \phi(z,w)d\mu_n(z)d\mu_n(w)\stackrel{n\rightarrow \infty}{\rightarrow}\int_K \int_K \phi(z,w)d\mu(z)d\mu(w)\notag
\end{eqnarray}
Now define $\phi (z,w)=\min(-\log|z-w|,m)$ for $m\geq 1$ and we have an estimate
\begin{eqnarray}
\liminf_{n\rightarrow\infty} I(\mu_n)&=& \liminf_{n\rightarrow\infty}\int_K\int_K -\log|z-w| d\mu_n(z) d\mu_n(w)\notag\\
&\geq &\liminf_{n\rightarrow\infty}\int_K\int_K \min(-\log|z-w|,m) d\mu_n(z) d\mu_n(w)\notag\\
&=& \int_K\int_K \min(-\log|z-w|,m) d\mu(z) d\mu(w)\notag
\end{eqnarray}
The last step is to apply the montone convergence theorem to the inequality just obtained with $m\rightarrow \infty$ to conclude the result.
\end{proof}

Having a guess, how we could construct the equilibrium measure provided by the previous lemma, the next result gives the answer for our initial question in the case of compact sets.

\begin{thm}
If $K$ is a compact subset of $\mathbb{C}$ then $K$ has an equlibrium measure.
\end{thm}

\begin{proof}
Observe that there exists a sequence of measures in $\cal{P}$$(K)$ such that 
\begin{eqnarray}
I(\mu_n)\rightarrow \inf_{\mu\in\text{$\cal{P}$$(K)$}}I(\mu)\notag
\end{eqnarray}
so by compactness of $K$ (and some elementary measure theory) there also exists a $\text{weak}^*$-convergent subsequence $\mu_{n_k}$ such that $\mu_{n_k}\rightarrow \nu$ for $k\rightarrow\infty$ and some probability measure $\nu$ on $K$, but now we can apply the previous lemma to deduce that 
\begin{eqnarray}
\liminf_{k\rightarrow \infty} I(\mu_{n_k}) \geq I(\nu)\notag
\end{eqnarray}  
So $\nu$ is an equilibrium measure.
\end{proof}

The result is relatively obvious from the viewpoint of electrostatics, but as we have seen, mathematically not immediate. The question of uniqueness for the equilibrium measures can answered as well, but requires far more tools than we have developed, so we content ourselves by stating the result.

\begin{fact}[\cite{4}]
\label{fact:emb}
Let $K$ be a compact non-polar subset of $\mathbb{C}$, then it has a unique equilibrium measure. Furthermore this measure $\nu$ is supported on the exterior boundary, i.e. $supp(\nu)\subset \partial_e K$.
\end{fact}

The preceeding fact allows us to at least deduce explicitly the equilibrium measure for the disc.

\begin{ex} 
\label{ex:lebb}
Let $\bar{\Delta (0,1)}=\bar{\Delta}$ denote the closed unit disc. Then since the equlibium measure $\nu$ is supported on $\partial \bar{\Delta}$ and unique, it must be rotation invariant. This implies that $\nu$ is the normalized Lebesgue measure on the boundary. 
\end{ex}

The only question, which remains from our initial heuristic derivation of the potential function, is the question whether in a state of equilibrium the potential is constant as it is predicted by experiment since non-constant potentials would introduce a force and hence no equilibrium could persist. It turns out that our intuition is also mathematically correct.

\begin{fact}[\index{Frostman's Theorem}Frostman's Theorem, \cite{4} or \cite{6}] Let $K$ be a compact subset of $\mathbb{C}$ and let $\nu$ denote its equilibrium measure, then $-p_\nu\leq I(\nu)$ on $\mathbb{C}$ and
\begin{eqnarray}
-p_\nu=I(\nu) \qquad \text{on $K\backslash E$, where $E$ is a polar subset of $\partial K$} 
\end{eqnarray}
\end{fact}

So except for polar sets on the boundary we have constant potentials on a set in equilibrium. Frostman's Theorem turns out to be of fundamental importance in potential theory and provides the tool for many advanced results especially in the context of harmonic measure (see \cite{14}).

\subsection{Capacity and Hausdorff Measure}

Measuring the size of sets by using the notion of being polar or non-polar turns out to be not sufficient in many circumstances and we shall proceed to another concept, which is highly important in potential theory. To give an application as motivation, we briefly recall the definition of Hausdorff measure. 

\begin{defn} 
Let $A$ be a bounded susbset of $\mathbb{C}$ and let $h(t)=t^p$ for $0<p<\infty$. Further let $B_i$ be (solid) squares in the complex plane parallel to the coordinate axes, then we define
\begin{eqnarray}
H_{p,\delta}(A)=\inf \{ \sum_{i=1}^\infty h(d_i): A\subset \bigcup_{i=1}^\infty B_i,\text{ } diam(B_i)\leq d_i \}
\end{eqnarray}
And then we set
\begin{eqnarray}
H_p=\lim_{\delta\rightarrow 0} H_{p,\delta}(A)
\end{eqnarray} 
which is called the \textbf{p-dimensional (outer) \index{Hausdorff measure}Hausdorff measure} of A.
\end{defn}

Notice that $H_{p,\delta}$ increases if $\delta$ decreases as we take an infimum over smaller and smaller families of sets. It is a classical fact that Hausdorff measure is not a 'measure'.

\begin{fact}[\cite{1}] 
$H_p$ is (only) an outer measure.
\end{fact}

Hausdorff measure is not countably additive in general. Nevertheless, one might ask, what are the sets of Hausdorff measure zero and how to distinguish them from sets of positive Hausdorff measure. So we ask\\

\textbf{Is there a (potential-theoretic) criterion to determine sets of Hausdorff measure zero?} \\

We shall develop another standard concept in potential theory to answer this question for a special case. We mainly follow \cite{6} and partly \cite{15}. 

\begin{defn} 
Let $E$ be a compact subset of $\mathbb{C}$ and let $\nu$ be the equilibrium measure for $E$, then we define the \textbf{\index{capacity}capacity} of $E$, denoted $c(E)$ by
\begin{eqnarray}
c(E)=e^{-I(\nu)}
\end{eqnarray} 
\end{defn}

It is clear from the definition that if $E_1\subseteq E_2$ are compact sets, then $c(E_1)\leq c(E_2)$. As a simple example we determine the capacity of the unit disc.

\begin{ex} 
We know from Example \ref{ex:lebb} that the equilibrium measure $\nu$ for $\bar{\Delta(0,1)}$ is given by the normalized Lebesgue measure on the boundary. Therefore we can calculate the energy directly, namely
\begin{eqnarray}
I(\nu)&=&-\int_{\partial \Delta} \int_{\partial \Delta} \log|z-w| d\nu(z) d\nu(w)=\frac{1}{4\pi}\int_0^{2\pi} \int_0^{2\pi} \log|e^{is}-e^{it}|ds dt\notag\\
&=& \frac{1}{4\pi}\int_0^{2\pi} \int_0^{2\pi} \log\sqrt{2-2\cos (s-t)}ds dt=\frac{i}{4}\int_0^{2\pi}\pi -t+2i\log(1-e^{it})-2i\log\sqrt{2-2\cos (t)}dt\notag\\
&=&\left[\frac{\pi i t}{4}\left(\pi-2i\left(-\frac{it}{2}+\log\sqrt{2-e^{it}-e^{-it}}-\log (1-e^{-it})   \right)\right)\right]_0^{2\pi}=0\notag
\end{eqnarray}
where we have taken appropiate limits in the last expression. Hence we obtain that $c(\bar{\Delta})=e^0=1$.
\end{ex} 

The preceeding example indicates that it is in general impossible to calculate the capacity of a set exactly, but there exist several good general theorems on estimates for capacity. Now we extend the definition of capacity to arbitrary subsets of $\mathbb{C}$, therefore making it comparable with Hausdorff measure in terms of domains.

\begin{defn} 
The \textbf{capacity} of any subset $E$ of $\mathbb{C}$ is defined by
\begin{eqnarray}
c(E)=\sup_{F\subset E} c(F) \qquad \text{$F$ compact} 
\end{eqnarray}
\end{defn} 

We have basically defined the notion of \index{inner logarithmic capacity}'inner logarithmic capacity' in the previous definition; a definition of \index{outer logarithmic capacity}'outer logarithmic capacity' obviously approximates a given set $E$ by open sets covering it. We also drop 'logarithmic' as we did for potentials and shall only use our defintion given above, which seems to be more concrete in any case.\\

We remark in passing that it is immediate that any polar set has capacity $0$ if we agree on the convention $e^{-\infty}=0$. The three most basic result about capacity are summarized in the next result.

\begin{thm}
\label{thm:capbasic}
Let $E_1$ and $E_2$ be subsets of $\mathbb{C}$, then
\begin{description}
	\item[(a)] if $E_1\subseteq E_2$ it follows that $c(E_1)\leq c(E_2)$
	\item[(b)] if $T:z\mapsto az+b$ for $a,b\in \mathbb{C}$ and $T(E_1)=E_2$ $\Rightarrow$ $c(E_2)=|a|c(E_1)$
	\item[(c)] if $E_1$ is compact, then $c(E_1)=c(\partial_e E_1)$  
\end{description}
\end{thm} 

\begin{proof}
For (a) we simply apply the definition and property (c) is immediate from Fact \ref{fact:emb}. For (b) note that $\mu$ is a measure on $E_1$ - i.e. $supp(\mu)\subseteq E_1$ - if and only if $supp(\mu T^{-1})\subseteq E_2=aE_1+b$. Hence we have 
\begin{eqnarray}
I(\mu T^{-1})=-\int \int \log|az+b-aw-b|d\mu(z)d\mu(w)=I(\mu)-\log|a| 
\end{eqnarray} 
and the result follows by taking the exponential to get the capacity.
\end{proof}

The following example shows that although capacity is a positive set function, it is not (!) a measure.

\begin{ex} Consider the closed unit disc $\bar{\Delta}$ and let $J_i$ be the closed line segments contained in $\bar{\Delta}$ and obtained from $[-1,1]$ by rotation around the origin with rational angles, so $\{J_i\}$ is a countable set. Note that since each $J_i$ is compact there exists an equilibrium measure with finite energy, so that $c(J_i)>0$. Rotation-invariance from the last theorem (see \ref{thm:capbasic} (b)) shows that $c(J_i)=c(J_k)$ for all $i,k$. But
\begin{eqnarray}
\bigcup_{i=1}^\infty J_i\subset \bar{\Delta} \qquad \text{and} \qquad 1=c(\bar{\Delta})<\sum_{i=1}^\infty c(J_i)=\infty
\end{eqnarray} 
so that $c$ is not countably additive. 
\end{ex}

Confident with the basics of capacity and a guess that it might be reasonable to compare it with Hausdorff measure we find a striking result.

\begin{thm} 
\label{thm:chc}
Suppose $E$ is a compact subset of $\mathbb{C}$. If $H_p(E)>0$ for a fixed $p\in(0,\infty)$, then $c(E)>0$. In particular it follows that if $c(E)=0$ then $H_p(E)=0$.
\end{thm}

The proof requires a basic lemma estimating Hausdorff measure on the unit square, which we only state as the proof is a simple subdivision and counting argument.

\begin{fact}[\cite{6}] 
\label{fact:hme}
Let $E$ be a compact subset of the unit square (side length 1) centered at $0$ in $\mathbb{C}$. Suppose $H_p(E)>0$ for a fixed $p\in(0,\infty)$ then there exists a finite Borel measure $\mu$ on $E$ such that for any $a\in E$
\begin{eqnarray}
\mu(\Delta(a,r))\leq 36r^p \qquad (0<r\leq1)
\end{eqnarray} 
\end{fact}

Using this fact we can proceed to the proof of the theorem.

\begin{proof}(of Theorem \ref{thm:chc}) First of all we can assume without loss of generality that $E$ is contained in the unit square centered at $0$. Let $\mu$ be the finite Borel measure provided by Fact \ref{fact:hme}. The idea is to obtain an estimate for the potential of $\mu$. Define $\Omega(r)=\mu(\Delta(z,r))$ for $z\in E$, then by the previous fact (see \ref{fact:hme} again) we have $\Omega(r)\leq 36 r^p$. Now we can estimate the potential
\begin{eqnarray}
-p_\mu(z)&=&\int_E -\log|z-w|d\mu(w)=\int_0^1 \log\left(\frac{1}{r}\right)d\Omega(r)\notag\\
&=&\left[-\Omega(r)\log r\right]_0^1+\int_0^1 \frac{\Omega(r)}{r} dr\leq \left[-\Omega(r)\log r\right]_0^1+36\int_0^1 \frac{r^p}{r} dr\notag
\end{eqnarray}
Since we notice the following two facts
\begin{eqnarray}
\int_0^1 \frac{r^p}{r} dr <\infty \qquad \text{and} \qquad \lim_{r\rightarrow 0} r^p\log \frac{1}{r}=0 \qquad \quad \text{as $0<p<\infty$} \notag
\end{eqnarray}
we deduce that $\lim_{r\rightarrow 0}\frac{\Omega(r)}{r}=0$ and we get a bound on the potential given by
\begin{eqnarray}
-p_\mu(z)\leq 36\int_0^1 \frac{r^p}{r}dr <\infty\notag
\end{eqnarray}
Therefore $I(\mu)<\infty$ and we conclude that $c(E)>0$.
\end{proof}

Notice that a criterion to decide when $H^p=0$ can turn out to be more useful than it seems at first since the sets of Hausdorff measure zero are important for Hausdorff dimension, which can simply be defined as
\begin{eqnarray}
\text{\index{Hausdorff dimension}Hausdorff dimension of a set $A$}=\inf\{p>0:H_p(A)=0\}
\end{eqnarray}
and Hausdorff dimension is an important tool in the analysis of fractals. We now close our discussion of the basic concepts.

\subsection{What remains...?}

Certainly many important applications and some basic concepts (most notably here: thin sets, fine topologies) are missing from the previous sections. Those topics had to be left out either due to space constraints or the necessity to introduce lengthy background knowledge from an application area. Nevertheless we list here some of the important applications with a few references.\\

An important area for recent applications of subharmonic funtions and potential theory is clearly complex dynamics. Especially measure theoretic tools like capacity turn out to be useful for investigating questions regarding dimensions of sets arising under iterations of complex-valued maps. A good starting point with several references can be found in \cite{4}.\\

Another area is approximation theory, more specifically orthogonal polynomials and more modern developments using 'weighted' potential theory. We have briefly seen that one can obtain estimates of for the growth of polynomials, when discussing Green's function and it turns out that far-reaching generalizations of these techniques exist. For the classical case of orthogonal polynomials \cite{21} provides a complete account, while for 'weighted' potentials and more recent developments \cite{5} is an excellent source.\\

It turns out that questions regarding polynomials, which are solvable using potential theory, also have consequences in numerical analysis. For example the convergence factor for certain types of conjugate gradient methods for the iterative solution of linear systems is directly related to Green's function. A survery article is provided by \cite{22}.\\

Also regrettably missing from our discussion are the relations to geometry. For example the uniformization theorem for Riemann surfaces admits a proof using approximation techniques involving subharmonic functions (see e.g. \cite{8} or \cite{24}). Also one can generalize subharmonic functions to several (complex) variables and obtain 'plurisubharmonic' functions, which play an important role in complex geometry. For pluripotential theory the only coherent source seems to be the monograph \cite{23}, whereas the applications to geometry are scattered over the literature.\\

The importance of the relations between probability theory and potential theory already have been mentioned. It is clearly a good starting point to see how the complete theory is developed from a probabilistic viewpoint using martingales and Brownian motion in \cite{10}. Basically all important links are covered by the monumental monograph \cite{17}, which is for the particularly brave.\\

Since we have focused on potential theory in the complex plane, it should be mentioned that there exist many applications to complex analysis using the the tools developed so far. For example, the Riemann mapping theorem has a potential-theoretic proof. Even more relevant is the Koebe one-quarter theorem as it seems there is no proof circumventing potential theory. The standard source here is \cite{6}, but one should also be aware of the fact that the monograph \cite{6} is older than 30 years and partly results are outdated (compare e.g. with \cite{14}).\\ 

And as a last subject, one has to list functional analysis and in particular spectral theory, where seemingly potential theory is used although the general impact seems to be small. Some basics are covered in \cite{4}.\\ 

\begin{appendix}
\section{Appendix}

\subsection{Literature Overview}

The first obstacle in using subharmonic functions and potential theory as a tool in other areas of mathematics is the required background knowledge. One needs to be acquainted with real analysis and measure theory and for the case of potential theory in the two-dimensional case also with the main tools of complex analysis. The following two books turned out to be very helpful in the process of writing this paper:

\begin{itemize}
	\item Gerald B. Folland, \textit{Real Analysis - Modern Techniques and Their Applications} (see \cite{1})
	\item Theodore W. Gamelin, \textit{Complex Analysis} (see \cite{8})
\end{itemize}

Folland's book is indeed useful 'cover-to-cover', i.e. every topic is somehow relevant to potential theory, whereas Gamelin covers more than needed, but it is easy to figure out, which topics to omit. Except for some parts from probability theory, we have the following convention on background knowledge:\\

\textit{Results and definitions assumed in this paper without reference can be found in \cite{1} and/or \cite{8}!}\\

Furthermore one should be able to benefit from elementary knowledge in partial differential equations and a sound understanding of probability and Brownian motion. The following two books cover more than needed:

\begin{itemize}
	\item Lawrence C. Evans, \textit{Partial Differential Equations} (see \cite{2})
	\item Richard Durrett, \textit{Probability: Theory and Examples} (see \cite{13})
\end{itemize}

Durrett covers most of the background knowledge from probability necessary for this paper.\\

For the core of the paper we have used Ransford's book \textit{"Potential Theory in the Complex Plane"} (see \cite{4}), which is highly recommended. Three other good sources for the case of two-dimensions are

\begin{itemize}
	\item Edward B. Saff and Vilmos Totik, \textit{Logarithmic Potentials with External Fields} (see \cite{5})
	\item M. Tsuji, \textit{Potential Theory in Modern Function Theory} (see \cite{6})
	\item W.K. Hayman and P. Kennedy, \textit{Subharmonic functions. Vol. 1} (see \cite{16})
\end{itemize}

All three monographs cover either more specialized material or cover much more than needed for certain applications. For a very brief sketch of the basics one might also consult the lecture notes by Martinez-Finkelshtein (see \cite{15}).\\ 

For potential theory of three or more dimensions there is a variety of textbooks available; we mention only four major sources here

\begin{itemize}
	\item L. L. Helms, \textit{Introduction to Potential Theory} (see \cite{7})
	\item Naum S. Landkof, \textit{Foundations of Modern Potential Theory} (see \cite{18})
	\item W.K. Hayman and P. Kennedy, \textit{Subharmonic functions. Vol. 1} (see \cite{16})
	\item John Wermer, \textit{Potential Theory} (see \cite{19}) 
\end{itemize}

All four texts are classical and the author of this paper found many parts, especially on Green's functions and the Dirichlet problem very helpful. A more modern treatment can be found in the lecture notes of Papadimitrakis (see \cite{3}). If one wants to persue the relations with probability theory further one could consult 

\begin{itemize}
	\item Olav Kallenberg, \textit{Foundations of Modern Probability} (see \cite{12})
	\item Sidney C. Port and Charles C. Stone, \textit{Brownian Motion and Classical Potential Theory} (see \cite{10})
	\item Joseph L. Doob, \textit{Classical Potential Theory and Its Probabilistic Counterpart} (see \cite{17})
\end{itemize}

where Kallenberg's book is a background reference. Port and Stone's exposition is quite dense, but highly recommended, whereas Doob's monumental monograph should be used as a reference. For a historical perspective on potential theory the main reference is

\begin{itemize}
	\item Oliver D. Kellogg, \textit{Foundations of Potential Theory} (see \cite{20})
\end{itemize}

which can be regarded as the major work of modern potential theory in the first half of the $20^{th}$-century. 

\end{appendix}

\printindex

\end{document}